\documentclass[preprint,12pt]{article}

\usepackage{amsmath,amsthm,amssymb}

\usepackage{graphbox}

\DeclareMathOperator{\sgn}{sgn}
\DeclareMathOperator{\arccot}{arccot}

\newcommand{\M}{\mathbf M}

\newtheorem{teo}{Theorem}[section]
\newtheorem{lema}[teo]{Lemma}

\begin{document}

\title{\LARGE A Pseudospectral Method for the One-Dimensional Fractional Laplacian on $\mathbb R$}

\author{\normalsize Jorge Cayama$^1$ \and \normalsize Carlota M. Cuesta$^1$ \and \normalsize Francisco de la Hoz$^2$}

\date{
	\footnotesize
	$^1$Department of Mathematics, Faculty of Science and Technology, University of the Basque Country UPV/EHU, Barrio Sarriena S/N, 48940 Leioa, Spain \\[1em]
	$^2$Department of Applied Mathematics and Statistics and Operations Research, Faculty of Science and Technology, University of the Basque Country UPV/EHU, Barrio~Sarriena~S/N, 48940 Leioa, Spain
}

\maketitle

\begin{abstract}
	
In this paper, we propose a novel pseudospectral method to approximate accurately and efficiently the fractional Laplacian without using truncation. More precisely, given a bounded regular function defined over $\mathbb R$, we map the unbounded domain into a finite one, then we represent the function as a trigonometrical series. Therefore, the central point of this paper is the computation of the fractional Laplacian of an elementary trigonometric function.
	
As an application of the method, we also do the simulation of Fisher's equation with fractional Laplacian in the monostable case.
	
\end{abstract}

\medskip

\noindent\textit{Keywords:}

\noindent Fractional Laplacian, Pseudospectral methods, rational Chebyshev functions, nonlocal Fisher's equation, accelerating fronts
	
\section{Introduction}

In this  paper we present a pseudospectral method that approximates the one-dimensional fractional Laplacian operator of smooth functions on $\mathbb{R}$, by mapping $\mathbb{R}$ to a finite interval, and, thus, avoiding truncation.

The fractional Laplacian operator, denoted as $(-\Delta)^{\alpha/2}$, is a generalization of the integer-order Laplacian $\Delta \equiv \partial^2/\partial x_1^2 + \ldots \partial^2/\partial x_d^2$, with $d$ being the dimension. It appears in a number of applications (see, for instance, \cite[Table 1]{lischke} and its references), and can be defined in different equivalent ways \cite{kwasnicki}. In our case, we consider the following definition on $\mathbb R$:
\begin{equation}
\label{fraclapl}
(-\Delta)^{\alpha/2}u(x) = c_\alpha\int_{-\infty}^\infty\frac{u(x)-u(x+y)}{|y|^{1+\alpha}}dy,
\end{equation}

\noindent where $\alpha\in(0,2)$, and
$$
c_\alpha = \alpha\frac{2^{\alpha-1}\Gamma(1/2+\alpha/2)}{\sqrt{\pi}\Gamma(1-\alpha/2)}.
$$

\noindent The equivalent definition in the Fourier side is
$$
[(-\Delta)^{\alpha}u]^\wedge(\xi) = |\xi|^\alpha \hat u(\xi),
$$

\noindent and, hence, when $\alpha=2$, we recover $-\Delta u(x) = -u_{xx}(x)$, whereas, when $\alpha=0$, $(-\Delta)^0u = u$. On the other hand, bearing in mind that the Hilbert transform \cite{hilbert}
$$
\mathcal H (u)(x) = \frac1\pi\int_{-\infty}^{\infty}\frac{u(y)}{x-y}dy
$$

\noindent is defined in the Fourier side as $[\mathcal H(u)]^\wedge(\xi) = -i\sgn(\xi)\hat u(\xi)$, it follows that, when $\alpha=1$, $[(-\Delta)^{1/2}u]^\wedge(\xi) = |\xi|\hat u(\xi) = (-i\sgn(\xi))(i\xi)\hat u(\xi)$, or, equivalently, $(-\Delta)^{1/2}u(x) = \mathcal H(u_x)(x)$.

Remark that, in \cite{kwasnicki}, the author indeed considers $-(-\Delta)^{\alpha/2}u(x)$ in the definition of the fractional Laplacian, to make it agree with the integer-order Laplacian, when $\alpha = 2$. The same sign convention is chosen in \cite{pozrikidis}, where an excellent and up-to-date introduction to the topic can be found.

In recent years, there has been an increasing interest in evolution equations that incorporate nonlocal operators and, in particular, nonlocal operators that resemble a fractional power of the Laplacian or derivatives of fractional order. There are many models where such operators appear, and there is also an intrinsic mathematical interest in analyzing and simulating such equations. The list of references is so vast that we will concentrate here in the case where the fractional Laplacian appears instead of the usual term of Brownian diffusion, and we will focus on the paradigm nonlinear equation for this diffusion type, namely, Fisher's equation with fractional Laplacian:
\begin{equation}
\partial_{t}u +(-\Delta)^{\alpha/2}u = f(u),\quad x\in\mathbb{R},\ \ t \geq 0,
\label{fishereq}
\end{equation}

\noindent where, generically, $f(u) = u(1 - u)$ is the so-called monostable nonlinearity, or $f(u)=u(1-u)(u-a)$, with $a\in(0,1)$, is the bistable nonlinearity. In the case of classical diffusion, with $\alpha = 2$, this is a paradigm equation for pattern forming systems and reaction-diffusion systems in general (see the classical references for the monostable case, \cite{kpp}, \cite{fisher}, \cite{wazwaz}, \cite{danilov}, \cite{polyanin}, etc., and for the bistable case, \cite{starmer}, \cite{keener}, \cite{fitzhugh}, \cite{nagumo2}, \cite{schlogl}, etc.). The nonlocal version (\ref{fishereq}) has been proposed as a reaction-diffusion system with anomalous diffusion (see \cite{MancinelliVergniVulpiani2002}, \cite{MancinelliVergniVulpiani2003} and \cite{BaeumerKovacsMeerschaert2008}). Some fundamental analytical results appear in \cite{CabreSire2014} and \cite{CabreSire2015} for more general nonlinear equations and in several dimensions. Our main interest here is to simulate (\ref{fishereq}) as an illustration of a problem that requires a very large spatial domain or the whole domain, when traveling wave 
solutions ensue, since these travel in one direction and they do so with a wave speed exponentially increasing in time in the monostable case (see \cite{Engler2010}, \cite{CabreRoquejoffre2009} and \cite{CabreRoquejoffre2013}). In this regard, we will contrast the numerical results with 
the analytical ones.
 
The structure of this paper is as follows. In Section \ref{s:ComputationFracLap}, we propose a novel method to compute accurately the fractional Laplacian \eqref{fraclapl} without using truncation. More precisely, we rewrite the Laplacian in a more suitable way, which requires at least $\mathcal C^2$ regularity; then, after mapping the original domain $\mathbb R$ to $[0,\pi]$ by using the change of variable $x = L\cos(s)$, with $L > 0$, $s\in[0,\pi]$, we expand $u(s)\equiv u(L\cos(s))$ in Fourier series, and, at its turn, obtain the Fourier series expansion of $(-\Delta)^{\alpha/2}(e^{iks})$, which constitutes the central part of this paper. We also show how to generate efficiently an operational matrix $\M_\alpha$ that can be applied to the coefficients of the Fourier expansion of $u(s)$, to approximate $(-\Delta)^{\alpha/2}(u(s))$ at the equally-spaced nodes
\begin{equation}
\label{e:nodessj}
s_j = \frac{\pi(2j+1)}{2N}, \quad 0\le j\le N-1.
\end{equation}
Later on, in Section \ref{s:numtests}, we test the proposed method for a couple of functions. Finally, in Section \ref{s:fisher} we apply our method to the numerical simulation of (\ref{fishereq}) in the monostable case.

To the best of our knowledge, the numerical computation of the fractional Laplacian without truncating the domain has not being done so far. However, the change of variable $x = L\cos(s)$ was applied successfully in \cite{delahozcuesta2016} to compute numerically a related nonlocal operator defined on the whole real line. More precisely, in \cite{delahozcuesta2016}, from which we get several useful ideas, $\partial_x\mathcal D^\alpha$ was considered on $\mathbb R$, where the operator $\mathcal D^\alpha$ can be regarded as a left-sided fractional derivative in the Caputo sense (see, for instance, \cite{Kilbas}), with integration taken from $-\infty$:
\begin{equation}
\label{e:Da}
\mathcal D^\alpha u(x) = \frac{1}{\Gamma(1 - \alpha)}\int_{-\infty}^x\frac{u_x(y)}{(x-y)^\alpha}dy, \quad\alpha\in(0, 1).
\end{equation}

\noindent After defining $u(s)\equiv u(L\cos(s))$, $\partial_x\mathcal D^\alpha(u(s))$ was approximated at the nodes $s_j$ in \eqref{e:nodessj} by the composite midpoint rule taken over the families of nodes
\begin{equation*}
s_l^{(m)} = \frac{\pi(2l+1)}{2^{m+1}N}, \quad 0 \le l \le 2^{m}N-1, \quad m = 1, 2, \ldots,
\end{equation*}
although, in practice, only the indices $l$ satisfying $2^{m-1}(2j+1) \le l \le 2^{m}N-1$ were used, denoting as $[\partial_x\mathcal D^\alpha]^{m}(u(s))$ the resulting approximation. Then, studying the errors of several functions with different types of decay and applying Richardson extrapolation \cite{richardson} to $[\partial_x\mathcal D^\alpha]^{m}(u(s))$, it was conjectured that
\begin{equation}
\label{e:assymptfracder}
\begin{split}
\|[\partial_x\mathcal D^\alpha]^{m} u(x) & - \partial_x\mathcal D^\alpha u(x)\|_{\infty}
\cr
& =
\frac{c_1(\alpha)}{m^{2 - \alpha}} + \frac{c_2(\alpha)}{m^{3 - \alpha}} + \frac{c_3(\alpha)}{m^{4 - \alpha}} + \frac{c_4(\alpha)}{m^{5 - \alpha}} + \frac{c_5(\alpha)}{m^{6 - \alpha}} + \ldots,
\end{split}
\end{equation}
 and, indeed, this formula yielded very accurate results, at least for the functions considered. Remark that, in practice, $u(s)$ was expanded in Fourier series, so the extrapolation was really applied over $[\partial_x\mathcal D^\alpha]^{m}(e^{iks})$, which enabled to create an operational matrix acting on the coefficients of the Fourier expansion of $u(s)$.

As we can see, the main difference between this paper and \cite{delahozcuesta2016} is the numerical computation of the corresponding nonlocal operator acting on a single Fourier mode $e^{iks}$. In this paper, we have not considered the extrapolation technique, because it appears to be more involved than in \cite{delahozcuesta2016}, and, on the other hand, the method that we are proposing here is, in our opinion, very accurate.

To the best of our knowledge, the use of spectral and pseudospectral methods for approximating the Fractional Laplacian is limited to a few instances in the literature: we remark the works \cite{IlicLiuTurnerAnh2005}, \cite{YangLiuTurner2010}, and \cite{Bueno-OrovioKayBurrage2012}, where, although a truncation of the domain is not explicitly given, the method relays on the approximation of the fractional Laplacian by an operator on a truncated domain.

\section{Computation of the fractional Laplacian for regular functions}

\label{s:ComputationFracLap}

In the following pages, we will develop a new method to approximate numerically \eqref{fraclapl}. However, instead of working directly with \eqref{fraclapl}, we will use the representation given by the following lemma, which requires boundedness and $\mathcal C^2$-regularity.

\subsection{Equivalent form of the fractional Laplacian for regular functions}

\begin{lema}
	Consider the twice continuous bounded function $u\in\mathcal C_b^2(\mathbb{R})$. If $\alpha\in [1,2)$, or $\alpha\in(0,1)$ and $\lim_{x\to\pm \infty}u_x(x)=0$, then
	\begin{equation}
	(-\Delta)^{\alpha/2}u(x)=\left\lbrace
	\begin{aligned}
	& \frac{1}{\pi}\int_{-\infty}^\infty \frac{u_{x}(y)}{x-y}dy, & \alpha = 1, \\
	& \frac{c_{\alpha}}{\alpha(1-\alpha)}\int_{-\infty}^\infty \frac{u_{xx}(y)}{|x-y|^{\alpha-1}}dy, & \alpha \neq 1.
	\end{aligned}
	\right.\label{fraclapl2}
	\end{equation}
\end{lema}

\begin{proof}

Let us express first (\ref{fraclapl}) as an integral over $[0, \ \infty)$:

\begin{align}
	(-\Delta)^{\alpha/2}u(x) & = c_{\alpha}\int_{0}^\infty \frac{u(x)-u(x-y)+u(x)-u(x+y)}{y^{1+\alpha}}dy
	\cr
	& = c_{\alpha}\int_{0}^\infty \int_{0}^{y}\frac{u_{x}(x-z)-u_{x}(x+z)}{y^{1+\alpha}}dz\ dy
	\cr
	& = c_{\alpha}\int_{0}^\infty \left[(u_{x}(x-z)-u_{x}(x+z))\int_{z}^\infty 
	\frac{1}{y^{1+\alpha}}dy\right]dz
	\cr
	& = \frac{c_{\alpha}}{\alpha}\int_{0}^\infty \frac{u_{x}(x-z)-u_{x}(x+z)}{z^{\alpha}}dz,
	\label{intlema1}
\end{align}

\noindent where we have changed the order of integration. We distinguish three cases. When $\alpha = 1$, $c_1 = 1/\pi$, so
\begin{align*}
(-\Delta)^{1/2}u(x) & = \frac{1}{\pi}\int_{0}^\infty \frac{u_{x}(x-y)-u_{x}(x+y)}{y}dy = \frac{1}{\pi}\int_{-\infty}^\infty \frac{u_{x}(y)}{x-y}dy,
\end{align*}

\noindent i.e., $(-\Delta)^{1/2}u(x)$ is precisely the Hilbert transform \cite{hilbert} of $u_x(x)$, as mentioned in the introduction. On the other hand, when $\alpha\in(1,2)$,
\begin{align*}
	(-\Delta)^{\alpha/2}u(x) & = \frac{c_{\alpha}}{\alpha}\int_{0}^\infty \frac{u_{x}(x-z)-u_{x}(x)+u_{x}(x)-u_{x}(x+z)}{z^{\alpha}}dz\\
	& = -\frac{c_{\alpha}}{\alpha}\int_{0}^\infty \int_{0}^{z}\frac{u_{xx}(x-y)+u_{xx}(x+y)}{z^{\alpha}}dy\ dz\\
	& = -\frac{c_{\alpha}}{\alpha}\int_{0}^\infty \left[(u_{xx}(x-y)+u_{xx}(x+y))
	\int_{y}^\infty \frac{1}{z^{\alpha}}dz\right]dy\\
	& = -\frac{c_{\alpha}}{\alpha(\alpha - 1)}\int_{0}^\infty 
	\frac{u_{xx}(x-y)+u_{xx}(x+y)}{y^{\alpha-1}}dy \\
	& = \frac{c_{\alpha}}{\alpha(1 - \alpha)}\int_{-\infty}^\infty 
	\frac{u_{xx}(x + y)}{|y|^{\alpha-1}}dy = \frac{c_{\alpha}}{\alpha(1 - \alpha)}\int_{-\infty}^\infty 
	\frac{u_{xx}(y)}{|x-y|^{\alpha-1}}dy.
\end{align*}

\noindent Finally, when $\alpha\in(0,1)$, this last formula also holds, although the deduction is slightly different, and 
	$\lim_{x \rightarrow\pm\infty}u_{x}(x)=0$ is required. Indeed, from (\ref{intlema1}),	
	\begin{align*}
	(-\Delta)^{\alpha/2}u(x) & = \frac{c_{\alpha}}{\alpha}\int_{0}^\infty \int_{z}^\infty \frac{u_{xx}(x-y)+u_{xx}(x+y)}{z^{\alpha}}dy\ dz\\
	& = \frac{c_{\alpha}}{\alpha}\int_{0}^\infty \left[(u_{xx}(x-y)+u_{xx}(x+y))
	\int_{0}^{y}\frac{1}{z^{\alpha}}dz\right]dy\\
	& = \frac{c_{\alpha}}{\alpha(1-\alpha)}\int_{0}^\infty 
	\frac{u_{xx}(x-y)+u_{xx}(x+y)}{y^{\alpha-1}}dy \\
	& = \frac{c_{\alpha}}{\alpha(1 - \alpha)}\int_{-\infty}^\infty 
	\frac{u_{xx}(y)}{|x-y|^{\alpha-1}}dy,
	\end{align*}
	
\noindent which completes the proof of the lemma.
\end{proof}

\subsection{Mapping $\mathbb R$ to a finite interval}

As with the computation of fractional derivatives in \cite{delahozcuesta2016}, one of the main difficulties of approximating numerically \eqref{fraclapl2} is the unboundedness of the spatial domain. Hence, we follow the same approach as in \cite{delahozcuesta2016}, i.e., mapping $\mathbb R$ to a finite interval, followed by the use of a series expansion in terms of Chebyshev polynomials. Among the possible mappings, we use the so-called algebraic map \cite{Boyd1987}:

\begin{equation}
\label{e:algebraicmap}
\xi = \dfrac{x}{\sqrt{L^2 + x^2}}\in[-1,1] \Longleftrightarrow x = \dfrac{L\xi}{\sqrt{1 - \xi^2}}\in\mathbb R,
\end{equation}

\noindent with $L > 0$. Then, we consider the Chebyshev polynomials of the first kind $T_k(\xi)$ over the new domain $\xi\in[-1,1]$:
\begin{equation}
\label{e:Tk}
T_k(\xi) \equiv \cos(k\arccos(\xi)), \quad \forall k\in\mathbb Z.
\end{equation}

\noindent Therefore, introducing \eqref{e:algebraicmap} into \eqref{e:Tk}, we have the so-called rational Chebyshev functions $TB_k(x)$ \cite{Boyd1987}, which form a basis set for $\mathbb R$:
\begin{equation}
\label{e:TBk}
TB_k(x) \equiv T_k\left(\dfrac{x}{\sqrt{L^2+x^2}}\right), \quad x\in\mathbb R, \quad \forall k\in\mathbb Z.
\end{equation}

\noindent The rational Chebyshev functions are very adequate to represent functions defined on $\mathbb R$ having different types of decay as $x\to\pm\infty$  (see \cite{delahozvadillo} for a comparison of them with Hermite functions and sinc functions). Remark that they are closely related to the Christov functions and the Higgins functions \cite{boyd1990}, which are very adequate for computing numerically the Hilbert transform (see for instance \cite{weideman1995} and \cite{boydxu2011}) of functions in $L^2(\mathbb R)$, although they are not sufficient for representing the functions that we are considering in this paper, which are not in $L^2(\mathbb R)$.

In practice, we do not work directly with $TB_k(x)$ or with $T_k(\xi)$, but rather with a Fourier series expansion. Hence, we consider yet another change of variable:
$$
x = L\cot(s)\in\mathbb R\Longleftrightarrow \xi = \cos(s)\in[-1,1] \Longleftrightarrow s = \arccos(\xi) \in[0,\pi],
$$

\noindent such that $T_k(\xi) = T_k(\cos(s)) = \cos(ks)$. Therefore, a series expansion in terms of Chebyshev polynomials or rational Chebyshev functions is equivalent to a cosine expansion.

In order to express \eqref{fraclapl2} in terms of $s\in[0,\pi]$, we need the following identities \cite{Boyd1987}:
\begin{equation}
\begin{split}
	u_{x}(x) & = - \frac{\sin^{2}(s)}{L}u_{s}(s), \\
	u_{xx}(x) & = \frac{\sin^{4}(s)}{L^{2}}u_{ss}(s) + \frac{2\sin^{3}(s)\cos(s)}{L^{2}}u_{s}(s),
	\label{uderivadas}
\end{split}
\end{equation}
 
\noindent where, with some abuse of notation, $u(s) \equiv u(x(s))$. Then, bearing in mind that $dx = -L\sin^{-2}(s)ds$, \eqref{fraclapl2} becomes
\begin{equation}
(-\Delta)^{\alpha/2}u(s)=\left\{
\begin{aligned}
& {-}\frac{1}{L\pi}\int_{0}^\pi \frac{u_s(\eta)}{\cot(s) - \cot(\eta)}d\eta, & \alpha = 1, \\
& \frac{c_{\alpha}}{L^\alpha\alpha(1-\alpha)} \\
& \cdot\int_0^\pi \frac{\sin^{2}(\eta)u_{ss}(\eta) + 2\sin(\eta)\cos(\eta)u_{s}(\eta)}{|\cot(s)-\cot(\eta)|^{\alpha-1}}d\eta, & \alpha \neq 1,
\end{aligned}
\right.\label{e:fraclap0pi}
\end{equation}

\noindent or, equivalently,
\begin{equation}
(-\Delta)^{\alpha/2}u(s)=\left\{
\begin{aligned}
& \frac{\sin(s)}{L\pi}\int_{0}^\pi \frac{\sin(\eta)u_s(\eta)}{\sin(s - \eta)}d\eta, & \alpha = 1, \\
& \frac{c_{\alpha}|\sin(s)|^{\alpha-1}}{L^\alpha\alpha(1-\alpha)} \\
& \cdot\int_0^\pi \frac{\sin^{\alpha}(\eta)(\sin(\eta)u_{ss}(\eta) + 2\cos(\eta)u_{s}(\eta))}{|\sin(s-\eta)|^{\alpha-1}}d\eta, & \alpha \neq 1.
\end{aligned}
\right.\label{e:fraclap0pi2}
\end{equation}

\subsection{Discretizing the mapped bounded domain}

We discretize the interval $s\in[0,\pi]$ in the nodes defined in \eqref{e:nodessj}:
$$
s_{j}=\frac{\pi(2j+1)}{2N}\Longleftrightarrow x_{j} = L\cot\left(\frac{\pi(2j+1)}{2N}\right)\Longleftrightarrow \xi_{j}=\cos\left(\frac{\pi(2j+1)}{2N}\right),
$$

\noindent such that $s_0 = \pi / (2N)$, $s_{N-1} = \pi - \pi / (2N)$, $s_{j+1} - s_j = \pi/N$, for all $j$. Therefore, we avoid evaluating \eqref{e:fraclap0pi} directly at $s = 0$ and $s = \pi$. Even if we do not use it in this paper, let us mention, for the sake of completeness, that it is also possible to introduce a spacial shift in $x$, i.e.,
\begin{equation}
\label{e:xxc}
x = x_c + L\cot(s);
\end{equation}

\noindent so
$$
x_{j} = x_c + L\cot(s_j) = x_c + L\cot\left(\frac{\pi(2j+1)}{2N}\right).
$$

\noindent In general, along this paper, whenever we write $u(x_j)$, we refer to $u(x)$ evaluated at $x = x_j$, whereas, we write $u(s_j)$ to refer to $u(x(s))$ evaluated at $s_j$. Therefore, with some abuse of notation, $u(s_j)\equiv u(x_j)$. Observe that the definition of $s_j$ in \eqref{e:nodessj} does not depend on $x_c$, whereas the definition of $x_j$ does, which makes preferable to work with $u(s_j)$, especially when $x_c\not=0$ is used. On the other hand, as mentioned above, since $s\in[0,\pi]$, a cosine series expansion is enough to represent $u(s)$. However, we rather consider a more general series expansion formed by $e^{iks}$, with $k\in\mathbb Z$, which is somehow easier to implement numerically:
$$
u(s) = \sum\limits_{k=-\infty}^\infty \hat{u}(k)e^{iks}, \qquad s\in[0,\ \pi].
$$

\noindent Hence, in order to determine the coefficients $\hat{u}(k)$, we have to extend the definition of $u(s)$ to $s\in[0,2\pi]$. Note that an even expansion of $u(s)$ at $s = \pi$ will yield precisely a cosine series.

Since it is impossible to work with infinitely many frequencies, we approximate $u(s)$ as
\begin{equation}
u(s) \approx \sum\limits_{k=-N}^{N-1}\hat{u}(k)e^{iks}, \qquad s\in[0,\ 2\pi].
\label{uenfourier2}
\end{equation}

\noindent Then, taking $0\le j\le 2N-1$ in \eqref{e:nodessj}, we adopt a pseudospectral approach (see \cite{trefethen} for a classical introduction to spectral methods in \textsc{Matlab}) to determine uniquely the $2N$ coefficients $\hat u(k)$ in \eqref{uenfourier2}, i.e., we impose \eqref{uenfourier2} to be an equality at $s_j$:
\begin{align}
u(s_j) & \equiv \sum\limits_{k=-N}^{N-1}\hat{u}(k)e^{iks_j}
= \sum\limits_{k=-N}^{N-1}\hat{u}(k)e^{ik\pi(2j+1)/(2N)}
	\cr
& =
\sum\limits_{k=0}^{2N-1}\left[\hat{u}(k)e^{ik\pi/(2N)}\right]e^{2ijk\pi/(2N)}.
\label{uenfourier3}
\end{align}

\noindent Equivalently, the $(2N)$-periodic coefficients $\hat u(k)$ are given by
\begin{equation}
\hat{u}(k) \equiv \frac{e^{-ik\pi/(2N)}}{2N}\sum\limits_{j=0}^{2N-1}u(s_{j})e^{-2ijk\pi/(2N)}.
\label{uenfourier4}
\end{equation}

\noindent Note that the discrete Fourier transforms (\ref{uenfourier3}) and (\ref{uenfourier4}) can be computed very efficiently by means of the fast Fourier Transform (FFT) \cite{FFT}. On the other hand, we apply systematically a Krasny filter \cite{krasny}, i.e., we set to zero all the Fourier coefficients $\hat u(k)$ with modulus smaller than a fixed threshold, which in this paper is the epsilon of the machine.

\subsection{An explicit calculation of $(-\Delta)^{\alpha/2}e^{iks}$}

Since we are approximating $u(s)$ by \eqref{uenfourier2}, the problem is reduced to computing $(-\Delta)^{\alpha/2}e^{iks}$. In this section, we will prove the following theorem:

\begin{teo}
\label{t:teo}

Let $\alpha\in(0,1)\cup(1,2)$, then
\begin{equation}
\label{e:thanot1}
(-\Delta)^{\alpha/2}(e^{iks})
=
\left\{
\begin{aligned}
& \frac{c_{\alpha}|\sin(s)|^{\alpha-1}}{8L^\alpha\tan(\frac{\pi\alpha}{2})}\sum_{l=-\infty}^\infty e^{i2ls}((1-\alpha)k^2-4kl)
\\
& \ \cdot \frac{\Gamma\left(\frac{-1+\alpha}{2}+|l|\right)\Gamma\left(\frac{-1-\alpha}{2}+\left|\frac{k}{2}-l\right|\right)}{\Gamma\left(\frac{3-\alpha}{2}+|l|\right)\Gamma\left(\frac{3+\alpha}{2}+\left|\frac{k}{2}-l\right|\right)}, & \text{$k$ even,}
\\
& i\frac{c_{\alpha}|\sin(s)|^{\alpha-1}}{8L^\alpha}\sum_{l=-\infty}^\infty e^{i2ls}((1-\alpha)k^2-4kl)
\\
& \ \cdot\sgn(\tfrac{k}{2}-l)\frac{\Gamma\left(\frac{-1+\alpha}{2}+|l|\right)\Gamma\left(\frac{-1-\alpha}{2}+\left|\frac{k}{2}-l\right|\right)}{\Gamma\left(\frac{3-\alpha}{2}+|l|\right)\Gamma\left(\frac{3+\alpha}{2}+\left|\frac{k}{2}-l\right|\right)}, & \text{$k$ odd.}
\end{aligned}
\right.
\end{equation}

\noindent Moreover, when $\alpha = 1$,
\begin{equation}
\label{e:tha1}
(-\Delta)^{1/2}(e^{iks}) =
\left\lbrace
\begin{aligned}
& \frac{|k|\sin^2(s)}{L}e^{iks}, & \text{$k$ even,}
	\\
& \frac{ik}{L\pi}\left(\frac{-2}{k^2-4} - \sum_{l=-\infty}^\infty\frac{4\sgn(l)e^{i2ls}}{(k-2l)((k-2l)^2-4)}\right), & \text{$k$ odd.}
\end{aligned}
\right.
\end{equation}

\end{teo}

\begin{proof}

We prove first the case $\alpha = 1$. Introducing $u(s) = e^{iks}$ in \eqref{e:fraclap0pi2}, we get
$$
(-\Delta)^{1/2}(e^{iks}) = \frac{ik\sin(s)}{L\pi}\int_{0}^\pi \frac{\sin(\eta)e^{ik\eta}}{\sin(s - \eta)}d\eta.
$$

\noindent When $k \equiv 0\bmod2$,
\begin{align*}
\int_{0}^{\pi} \frac{\sin(\eta)e^{ik\eta}}{\sin(s - \eta)}d\eta = -e^{iks}\cos(s)\int_{0}^{\pi}e^{ik\eta}d\eta - e^{iks}\sin(s)\int_{0}^{\pi} \frac{\cos(\eta)e^{ik\eta}}{\sin(\eta)}d\eta.
\end{align*}

\noindent The first integral is trivial, and the second can be calculated explicitly, too:
\begin{align*}
\int_0^\pi\frac{\cos(\eta)e^{ik\eta}}{\sin(\eta)}d\eta & = \frac{i\sgn(k)}2\int_0^{2\pi}\frac{\cos(\eta)\sin(|k|\eta)}{\sin(\eta)}d\eta
	\cr
& =
\begin{cases}
0, & k = 0,
	\\
i\pi\sgn(k), & k\in2\mathbb Z\backslash\{0\};
\end{cases}
\end{align*}

\noindent which is easily proved by induction on $2\mathbb N$, bearing in mind that $\sin(2\eta) = 2\sin(\eta)\cos(\eta)$, and that $\sin((|k|+2)\eta) - \sin(|k|\eta) = 2\sin(\eta)\cos((|k|+1)\eta)$. Therefore,
\begin{equation}
\label{e:int0picotexp}
\int_{0}^{\pi} \frac{\sin(\eta)e^{ik\eta}}{\sin(s - \eta)}d\eta =
\begin{cases}
-\pi\cos(s), & k = 0,
\\
-i\pi\sgn(k)\sin(s)e^{iks}, & k\in2\mathbb Z\backslash\{0\};
\end{cases}
\end{equation}

\noindent from which follows the first part of \eqref{e:tha1}. On the other hand, when $k\equiv1\bmod2$, $e^{ik\eta}$ is not periodic in $\eta\in[0,\pi]$, and there seems to be no compact formula for $(-\Delta)^{1/2}(e^{iks})$, as in $k\equiv0\bmod2$. Hence, we have to consider a series expansion for $(-\Delta)^{1/2}(e^{iks})$; more precisely, we write
\begin{equation}
\label{e:seriesa1}
\sin(s)\int_{0}^{\pi}\frac{\sin(\eta)e^{ik\eta}}{\sin(s - \eta)}d\eta = \sum_{l=-\infty}^\infty c_{kl}e^{i2ls},
\end{equation}

\noindent with $c_{kl}$ given by
\begin{align*}
c_{kl} & = \frac1\pi\int_{0}^\pi\left[\sin(s)\int_{0}^{\pi}\frac{\sin(\eta)e^{ik\eta}}{\sin(s - \eta)}d\eta\right]e^{-i2ls}ds
	\cr
& = \frac1\pi\int_{0}^\pi\sin(\eta)e^{ik\eta}\left[\int_0^\pi\frac{\sin(s)e^{-i2ls}}{\sin(s - \eta)}ds\right]d\eta,
\end{align*}

\noindent where we have changed the order of integration. The inner integral is given by \eqref{e:int0picotexp}:
$$
\int_{0}^{\pi} \frac{\sin(s)e^{-i2ls}}{\sin(s - \eta)}ds =
\begin{cases}
\pi\cos(\eta), & l = 0,
\\
-i\pi\sgn(l)\sin(\eta)e^{-i2l\eta}, & l\in\mathbb Z\backslash\{0\}.
\end{cases}
$$

\noindent Hence,
$$
c_{kl} =
\left\{
\begin{aligned}
& \int_{0}^\pi\sin(\eta)\cos(\eta)e^{ik\eta}d\eta = \frac{-2}{k^2 - 4}, & & l = 0,
	\\
& {-i}\sgn(l)\int_{0}^\pi\sin^2(\eta)e^{i(k-2l)\eta}d\eta = \frac{-4\sgn(l)}{(k-2l)((k-2l)^2-4)}, & & l\not=0,
\end{aligned}
\right.
$$

\noindent from which we conclude the second part of \eqref{e:tha1}.

We consider now $\alpha\not=1$. Introducing $u(s) = e^{iks}$ in \eqref{e:fraclap0pi2}, we get
\begin{equation}
\label{e:fraclapeiksanot1}
(-\Delta)^{1/2}(e^{iks}) = \frac{c_{\alpha}|\sin(s)|^{\alpha-1}}{L^\alpha\alpha(1-\alpha)}\int_0^\pi \frac{\sin^{\alpha}(\eta)(-k^2\sin(\eta) + 2ik\cos(\eta))e^{ik\eta}} {|\sin(s-\eta)|^{\alpha-1}}d\eta.
\end{equation}

\noindent Then, as in \eqref{e:seriesa1}, we consider a series expansion:
\begin{equation}
\label{e:seriesdkl}
\int_{0}^{\pi}\frac{\sin^{\alpha}(\eta)(-k^{2}\sin(\eta) + 2ik\cos(\eta))e^{ik\eta}}{|\sin(s-\eta)|^{\alpha-1}}d\eta = \sum_{l=-\infty}^\infty d_{kl}e^{i2ls},
\end{equation}

\noindent with $d_{kl}$ given by
\begin{align}
\label{dkl}
d_{kl} & = \frac{1}{\pi}\int_{0}^{\pi}\left[\int_0^\pi\frac{\sin^{\alpha}(\eta)(-k^{2}\sin(\eta) + 2ik\cos(\eta))e^{ik\eta}}{|\sin(s-\eta)|^{\alpha-1}}d\eta \right]e^{-i2ls}ds
	\cr
& = \frac{1}{\pi}\int_{0}^{\pi}\sin^{\alpha}(\eta)(-k^{2}\sin(\eta) + 2ik\cos(\eta))e^{ik\eta}\left[\int_0^\pi\frac{e^{-i2ls}}{|\sin(s-\eta)|^{\alpha-1}}ds\right]d\eta
	\cr
& = \frac{1}{\pi}\left[\int_0^\pi\frac{e^{-i2ls}}{\sin^{\alpha-1}(s)}ds\right]\left[\int_{0}^{\pi}\sin^{\alpha}(\eta)(-k^{2}\sin(\eta) + 2ik\cos(\eta))e^{i(k-2l)\eta}d\eta\right]
	\cr
& = \frac1\pi I_1\cdot I_2,
\end{align}

\noindent where we have changed again the order of integration. Integrals of the type of $I_1$ and $I_2$ can be explicitly calculated by means of standard complex-variable techniques (see for instance \cite[p. 158]{nielsen}, for a classic reference). On the other hand, we have used \textsc{Mathematica} \cite{mathematica}, which computes them immediately (after, occasionally, very minor rewriting). The expression for $I_1$ is
\begin{align}
\label{e:I1}
I_1 & = \frac{e^{-i2\pi l}((2i)^\alpha + (-2i)^\alpha e^{i2\pi l})\pi\csc(\pi\alpha)\Gamma\left(\frac{-1+\alpha-2l}{2}\right)}{4\Gamma(-1+\alpha)\Gamma\left(\frac{3-\alpha-2l}{2}\right)}
	\cr
& = -\frac{2^{\alpha-1}\cos(\frac{\pi\alpha}{2})\Gamma(2-\alpha)\Gamma\left(\frac{-1+\alpha}{2}-l\right)}{\Gamma\left(\frac{3-\alpha}{2}-l\right)},
\end{align}

\noindent where we have used the well-known Euler's reflection formula $\Gamma(z)\Gamma(1-z) = \pi / \sin(\pi z)$. Moreover, applying twice Euler's reflection formula, 
\begin{align}
\label{e:reflection}
\frac{\Gamma(z)}{\Gamma(w)} = \frac{\Gamma(z)\Gamma(1-z)\Gamma(1-w)}{\Gamma(w)\Gamma(1-w)\Gamma(1-z)} = \frac{\sin(\pi w)}{\sin(\pi z)}\frac{\Gamma(1-w)}{\Gamma(1-z)}.
\end{align}

\noindent Therefore, for $l\in\mathbb Z$,
\begin{equation}
\label{e:reflectionapplied1}
\frac{\Gamma\left(\frac{-1+\alpha}{2}-l\right)}{\Gamma\left(\frac{3-\alpha}{2}-l\right)} = \frac{\sin\left(\pi\left(\frac{3-\alpha}{2}-l\right)\right)}{\sin\left(\pi\left(\frac{-1+\alpha}{2}-l\right)\right)}\frac{\Gamma\left(1-\left(\frac{3-\alpha}{2}-l\right)\right)}{\Gamma\left(1-\left(\frac{-1+\alpha}{2}-l\right)\right)}
= \frac{\Gamma\left(\frac{-1+\alpha}{2}+l\right)}{\Gamma\left(\frac{3-\alpha}{2}+l\right)},
\end{equation}

\noindent so the value of $I_1$ does not depend on the sign of $l$, and we can replace the appearances of $l$ in \eqref{e:I1} by $-l$, $|l|$ or $-|l|$. In this paper, we consider the last option, getting
\begin{equation}
\label{e:I1absl}
I_1 = -\frac{2^{\alpha-1}\cos(\frac{\pi\alpha}{2})\Gamma(2-\alpha)\Gamma\left(\frac{-1+\alpha}{2}+|l|\right)}{\Gamma\left(\frac{3-\alpha}{2}+|l|\right)},
\end{equation}

\noindent which is more convenient from an implementation point of view, as we will explain in Section \ref{s:numimpl}. Likewise, the expression for $I_2$ is
\begin{align*}
I_2 & = -2^{-2-\alpha}e^{-i\pi(\alpha+4l)/2}((-1)^k+e^{i\pi(\alpha+2l)})
	\\
& \ \cdot\frac{k((-1+\alpha)k+4l)\pi\csc(\pi\alpha)\Gamma\left(\frac{-1-\alpha+k-2l}{2}\right)}{\Gamma(-\alpha)\Gamma\left(\frac{3+\alpha+k-2l}{2}\right)}
	\\
& =
\left\{
\begin{aligned}
& \frac{\pi\alpha(1-\alpha)((-1+\alpha)k^2+4kl)\Gamma\left(\frac{-1-\alpha}{2}+\frac{k}{2}-l\right)}{2^{2+\alpha}\sin(\frac{\pi\alpha}{2})\Gamma(2-\alpha)\Gamma\left(\frac{3+\alpha}{2}+\frac{k}{2}-l\right)}, & & \text{$k$ even,}
\\
& i\frac{\pi\alpha(1-\alpha)((-1+\alpha)k^2+4kl)\Gamma\left(\frac{-1-\alpha}{2}+\frac{k}{2}-l\right)}{2^{2+\alpha}\cos(\frac{\pi\alpha}{2}) \Gamma(2-\alpha)\Gamma\left(\frac{3+\alpha}{2}+\frac{k}{2}-l\right)}, & & \text{$k$ odd.}
\end{aligned}
\right.
\end{align*}

\noindent Then, applying again \eqref{e:reflection}, we get expressions similar to \eqref{e:reflectionapplied1}:
$$
\frac{\Gamma\left(\frac{-1-\alpha}{2}+\frac{k}{2}-l\right)}{\Gamma\left(\frac{3+\alpha}{2}+\frac{k}{2}-l\right)} =
\left\{
\begin{aligned}
& \frac{\Gamma\left(\frac{-1-\alpha}{2}-\left(\frac{k}{2}-l\right)\right)}{\Gamma\left(\frac{3+\alpha}{2}-\left(\frac{k}{2}-l\right)\right)}, & & \text{$k$ even,}
\\
& {-}\frac{\Gamma\left(\frac{-1-\alpha}{2}-\left(\frac{k}{2}-l\right)\right)}{\Gamma\left(\frac{3+\alpha}{2}-\left(\frac{k}{2}-l\right)\right)}, & & \text{$k$ odd.}
\end{aligned}
\right.
$$

\noindent Hence, we obtain an equivalent but more convenient expression of $I_2$, containing absolute values as in \eqref{e:I1absl}:
\begin{align}
\label{e:I2abs}
I_2 & =
\left\{
\begin{aligned}
& \frac{\pi\alpha(1-\alpha)((-1+\alpha)k^2+4kl)\Gamma\left(\frac{-1-\alpha}{2}+\left|\frac{k}{2}-l\right|\right)}{2^{2+\alpha}\sin(\frac{\pi\alpha}{2})\Gamma(2-\alpha)\Gamma\left(\frac{3+\alpha}{2}+\left|\frac{k}{2}-l\right|\right)}, & & \text{$k$ even,}
\\
& i\sgn(\tfrac{k}{2}-l)\frac{\pi\alpha(1-\alpha)((-1+\alpha)k^2+4kl)\Gamma\left(\frac{-1-\alpha}{2}+\left|\frac{k}{2}-l\right|\right)}{2^{2+\alpha}\cos(\frac{\pi\alpha}{2}) \Gamma(2-\alpha)\Gamma\left(\frac{3+\alpha}{2}+\left|\frac{k}{2}-l\right|\right)}, & & \text{$k$ odd.}
\end{aligned}
\right.
\end{align}

\noindent Putting \eqref{e:I1absl} and \eqref{e:I2abs} together,
\begin{equation}
d_{kl} =
\left\{
\begin{aligned}
& \cot(\tfrac{\pi\alpha}{2})\alpha(1-\alpha)((1-\alpha)k^2-4kl)
	\\
& \ \cdot\frac{\Gamma\left(\frac{-1+\alpha}{2}+|l|\right)\Gamma\left(\frac{-1-\alpha}{2}+\left|\frac{k}{2}-l\right|\right)}{8\Gamma\left(\frac{3-\alpha}{2}+|l|\right)\Gamma\left(\frac{3+\alpha}{2}+\left|\frac{k}{2}-l\right|\right)}, & & \text{$k$ even},
	\\
& i\sgn(\tfrac{k}{2}-l)\alpha(1-\alpha)((1-\alpha)k^2-4kl)
	\\
& \ \cdot\frac{\Gamma\left(\frac{-1+\alpha}{2}+|l|\right)\Gamma\left(\frac{-1-\alpha}{2}+\left|\frac{k}{2}-l\right|\right)}{8\Gamma\left(\frac{3-\alpha}{2}+|l|\right)\Gamma\left(\frac{3+\alpha}{2}+\left|\frac{k}{2}-l\right|\right)}, & & \text{$k$ odd}.
\end{aligned}
\right.
\end{equation}

\noindent Therefore, bearing in mind \eqref{e:fraclapeiksanot1} and \eqref{e:seriesdkl}, we get \eqref{e:thanot1}, which concludes the proof of the theorem.

\end{proof}

\noindent Remark: under the change of variable $x = \cot(s)$, the cosine-like and sine-like Higgins functions \cite{boyd1990} are precisely $\cos(2ks)$ and $\sin((2k+2)s)$, which are eigenfunctions of the Hilbert transform \cite{weideman1995}. Therefore, the first part of \eqref{e:tha1} follows also from the results in \cite{weideman1995}.

\subsection{Constructing an operational matrix}

\label{s:numimpl}

As explained above, in order to compute $(-\Delta)^{\alpha/2}u(x)$ for a given function $u(x)$, we first represent it as \eqref{uenfourier2}, then we apply Theorem \ref{t:teo} to each basic function $e^{iks}$. In this paper, we have opted for a matrix approach, i.e., we have constructed a differencing matrix $\M_\alpha\in\mathcal M_{(2N)\times(2N)}(\mathbb C)$ based on Theorem \ref{t:teo}, such that
\begin{equation}
\label{e:Ma}
\begin{pmatrix}
(-\Delta)^{\alpha/2} u(s_0) \\ \vdots \\ (-\Delta)^{\alpha/2} u(s_{2N-1})
\end{pmatrix}
 \approx \M_\alpha\cdot
\begin{pmatrix}
\hat u(0) \\ \vdots \\ \hat u(N-1) \\ \hat u(-N) \\ \vdots \\ \hat u(-1)
\end{pmatrix},
\end{equation}

\noindent where the nodes $s_j$ are defined in \eqref{e:nodessj}. It is vital to underline that, by choosing the appropriate strategy, the speed in the construction of $\M_\alpha$, and therefore, in the numerical computation of $(-\Delta)^{\alpha/2}u(x)$, can be increased by several orders of magnitude. Furthermore, the matrix needs to be computed just once, and then it can be reused whenever needed.

In order to generate $\M_\alpha$, we compute $(-\Delta)^{\alpha/2}(e^{iks})$ according to Theorem \ref{t:teo}, for $k \in \{-N,\ldots N-1\}$. Moreover, from \eqref{e:fraclap0pi2}, $(-\Delta)^{\alpha/2}(e^{i0s}) = (-\Delta)^{\alpha/2}(1) = 0$, and
$$
\overline{(-\Delta)^{\alpha/2}(e^{iks})} = (-\Delta)^{\alpha/2}(e^{-iks}),
$$

\noindent so we only need to calculate the cases with $k > 0$. Finally, bearing in mind that $\hat u(-N) = \hat u(N)$, but $(-\Delta)^{\alpha/2}(e^{-iNs})\not=(-\Delta)^{\alpha/2}(e^{iNs})$, we impose $(-\Delta)^{\alpha/2}(e^{-iNs}) = 0$, so the obtention of $\M_\alpha$ is reduced to considering $k\in\{1, \ldots, N-1\}$.

Note that the implementation of Theorem \ref{t:teo} offers two difficulties: the need to evaluate the gamma function a very large number of times when $\alpha\not=0$, and the fact that $l$ is taken all over $\mathbb Z$.

With respect to the gamma function, a fast and accurate implementation is usually available in every major scientific environment, such as \textsc{Matlab} \cite{matlab}, which we use in this paper. More precisely, in \textsc{Matlab}, it is computed by the command \texttt{gamma}, which is based on algorithms outlined in \cite{gammamatlab}. However, using solely \texttt{gamma} to evaluate \eqref{e:thanot1} is not numerically stable, because of the quick growth of \texttt{gamma} (for instance, \texttt{gamma(172)} yields infinity); therefore, even for rather small values of $l$, we get spurious \texttt{NaN} results, because we are dividing infinity by infinity. One possible solution would be to use the command \textsc{Matlab} \texttt{gammaln}, which computes the natural logarithm of the gamma function, $\ln\Gamma$, i.e., the so-called log-gamma function:
\begin{align*}
& \frac{\Gamma\left(\frac{-1+\alpha}{2}+|l|\right)\Gamma\left(\frac{-1-\alpha}{2}+\left|\frac{k}{2}-l\right|\right)}{\Gamma\left(\frac{3-\alpha}{2}+|l|\right)\Gamma\left(\frac{3+\alpha}{2}+\left|\frac{k}{2}-l\right|\right)} \equiv \exp\left[\ln\Gamma\left(\tfrac{-1+\alpha}{2}+|l|\right) \right.
	\cr
& \qquad + \left.\ln\Gamma\left(\tfrac{-1-\alpha}{2}+\left|\tfrac{k}{2}-l\right|\right) - \ln\Gamma\left(\tfrac{3-\alpha}{2}+|l|\right) - \ln\Gamma\left(\tfrac{3+\alpha}{2}+\left|\tfrac{k}{2}-l\right|\right)\right];
\end{align*}

\noindent bear in mind that \texttt{gammaln} is not defined for negative values, so minor rewriting would be necessary in a few cases. However, in general, a much more convenient solution is to use the basic property $\Gamma(z+1) = z\Gamma(z)$:
\begin{equation}
\label{e:gammarecursive}
\begin{split}
\frac{\Gamma\left(\frac{-1+\alpha}{2}+|l|\right)}{\Gamma\left(\frac{3-\alpha}{2}+|l|\right)} & \equiv \frac{\frac{-3+\alpha}{2}+|l|}{\frac{1-\alpha}{2}+|l|}
\cdot\frac{\Gamma\left(\frac{-3+\alpha}{2}+|l|\right)}{\Gamma\left(\frac{1-\alpha}{2}+|l|\right)}
	\cr
\frac{\Gamma\left(\frac{-1-\alpha}{2}+\left|\frac{k}{2}-l\right|\right)}{\Gamma\left(\frac{3+\alpha}{2}+\left|\frac{k}{2}-l\right|\right)} & \equiv \frac{\frac{-3-\alpha}{2}+\left|\frac{k}{2}-l\right|}{\frac{1+\alpha}{2}+\left|\frac{k}{2}-l\right|}
\cdot\frac{\Gamma\left(\frac{-3-\alpha}{2}+\left|\frac{k}{2}-l\right|\right)}{\Gamma\left(\frac{1+\alpha}{2}+\left|\frac{k}{2}-l\right|\right)},
\end{split}
\end{equation}

\noindent where we consider separately the expressions containing $|l|$, and those containing $|k/2-l|$, because, for $|l|\gg1$,
$$
\frac{\frac{-3+\alpha}{2}+|l|}{\frac{1-\alpha}{2}+|l|} \approx 1, \qquad \frac{\frac{-3-\alpha}{2}+\left|\frac{k}{2}-l\right|}{\frac{1+\alpha}{2}+\left|\frac{k}{2}-l\right|} \approx 1,
$$

\noindent so the factorizations in \eqref{e:gammarecursive} are extremely stable from a numerical point of view. We apply recursively \eqref{e:gammarecursive}, until $|l| = 0$, and $|k/2-l| = 0$ (if $k$ even) or $|k/2-l|=1/2$ (if $k$ odd). Therefore, for any $l$ and $k$, the evaluations of $\Gamma$ needed to compute the left-hand sides of \eqref{e:gammarecursive} are just those in the quotients $\Gamma((-1+\alpha)/2) / \Gamma((3-\alpha)/2)$, $\Gamma((-1-\alpha)/2) / \Gamma((3+\alpha)/2)$ (if $k$ even), and $\Gamma(-\alpha/2) / \Gamma(2+\alpha/2)$ (if $k$ odd). Hence, taking into account that $\Gamma$ appears also in the definition of $c_\alpha$, it follows that the global number of evaluations of $\Gamma$ needed to compute \eqref{e:thanot1} is very small, although, unfortunately, it does not seem possible to remove completely all the evaluations of $\Gamma$.

Bearing in mind the previous arguments, in order to approximate \eqref{e:thanot1}, we precompute recursively the right-hand sides of \eqref{e:gammarecursive} for a large enough number of values, then store them in their respective vectors:
\begin{equation}
\label{e:recursion}
\begin{split}
\frac{\Gamma\left(\frac{-1+\alpha}{2}+|l|\right)}{\Gamma\left(\frac{3-\alpha}{2}+|l|\right)} & \equiv \frac{\Gamma\left(\frac{-1+\alpha}{2}\right)}{\Gamma\left(\frac{3-\alpha}{2}\right)}\prod_{m=0}^{|l|-1}\frac{\frac{-1+\alpha}{2}+m}{\frac{3-\alpha}{2}+m}, \quad \forall |l|\in\mathbb N,
	\cr
\frac{\Gamma\left(\frac{-1-\alpha}{2}+|\tilde l|\right)}{\Gamma\left(\frac{3+\alpha}{2}+|\tilde l|\right)} & \equiv \frac{\Gamma\left(\frac{-1-\alpha}{2}\right)}{\Gamma\left(\frac{3+\alpha}{2}\right)}\prod_{m=0}^{|\tilde l|-1}\frac{\frac{-1-\alpha}{2}+m}{\frac{3+\alpha}{2}+m}, \quad \forall |\tilde l|\in\mathbb N,
	\cr
\frac{\Gamma\left(\frac{-\alpha}{2}+|\tilde l|\right)}{\Gamma\left(2+\frac{\alpha}{2}+|\tilde l|\right)} & \equiv \frac{\Gamma\left(\frac{-\alpha}{2}\right)}{\Gamma\left(2+\frac{\alpha}{2}\right)}\prod_{m=0}^{|\tilde l|-1}\frac{\frac{-\alpha}{2}+m}{2+\frac{\alpha}{2}+m}, \quad \forall |\tilde l|\in\mathbb N,
\end{split}
\end{equation}

\noindent where the second expression is used in the cases with $k$ even, and the third one, in the cases with $k$ odd. Remark that the usage of absolute values makes trivial the computational evaluation of the vectors thus generated, because $|l|$, $|k/2 - l|$ (if $k$ even), and $|k/2 - l| + 1/2$ (if $k$ odd) are precisely their respective indices.

With respect to $l$ spanning $\mathbb Z$, we decompose it as $l = l_1N+l_2$, with $l_1\in\mathbb Z$, and $l_2\in\{-N/2,\ldots,N/2-1\}$; note that we take $l_2$ between $-N/2$ and $N/2-1$, rather than between $0$ and $N-1$, because the numerical results appear to be slightly more accurate in that way. Then, we observe that
$$
e^{i2ls_j} = e^{i2(l_1N+l_2)\pi(2j+1)/(2N)} = (-1)^{l_1}e^{i2l_2s_j},
$$

\noindent i.e., aliasing occurs when evaluating $e^{i2ls}$ in the actual nodes. Therefore, we truncate $l_1$, i.e., take $l_1\in\{-l_{lim}, \ldots, l_{lim}\}$, for $l_{lim}$ a large nonnegative integer. Then, \eqref{e:thanot1} becomes
\begin{equation}
\label{e:thanot1truncated}
(-\Delta)^{\alpha/2}(e^{iks_j})
\approx
\left\{
\begin{aligned}
& \frac{c_{\alpha}|\sin(s_j)|^{\alpha-1}}{8L^\alpha\tan(\frac{\pi\alpha}{2})}\sum_{l_2=-N/2}^{N/2-1}\Bigg[
\sum_{l_1=-l_{lim}}^{l_{lim}} a_{k,l_1,l_2}\Bigg]e^{i2l_2s_j}, & \text{$k$ even,}
\\
& i\frac{c_{\alpha}|\sin(s_j)|^{\alpha-1}}{8L^\alpha}\sum_{l_2=-N/2}^{N/2-1}\Bigg[\sum_{l_1=-l_{lim}}^{l_{lim}} a_{k,l_1,l_2}\Bigg]e^{i2l_2s_j}, & \text{$k$ odd,}
\end{aligned}
\right.
\end{equation}

\noindent where
\begin{equation}
a_{k,l_1,l_2} = 
\left\{
\begin{aligned}
& (-1)^{l_1}((1-\alpha)k^2-4k(l_1N+l_2))
\\
& \ \cdot \frac{\Gamma\left(\frac{-1+\alpha}{2}+|l_1N+l_2|\right)\Gamma\left(\frac{-1-\alpha}{2}+\left|\frac{k}{2}-l_1N-l_2\right|\right)}{\Gamma\left(\frac{3-\alpha}{2}+|l_1N+l_2|\right)\Gamma\left(\frac{3+\alpha}{2}+\left|\frac{k}{2}-l_1N-l_2\right|\right)}, & \text{$k$ even,}
\\
& (-1)^{l_1}((1-\alpha)k^2-4k(l_1N+l_2))\sgn(\tfrac{k}{2}-l_1N-l_2)
\\
& \ \cdot  \frac{\Gamma\left(\frac{-1+\alpha}{2}+|l_1N+l_2|\right)\Gamma\left(\frac{-1-\alpha}{2}+\left|\frac{k}{2}-l_1N-l_2\right|\right)}{\Gamma\left(\frac{3-\alpha}{2}+|l_1N+l_2|\right)\Gamma\left(\frac{3+\alpha}{2}+\left|\frac{k}{2}-l_1N-l_2\right|\right)}, & \text{$k$ odd.}
\end{aligned}
\right.
\end{equation}

\noindent In this way, since we have used \eqref{e:recursion} to precompute the appearances of $\Gamma$ and have stored them in three vectors, the computation of
$$
\sum_{l_1=-l_{lim}}^{l_{lim}} a_{k,l_1,l_2}
$$

\noindent is reduced to sums and products and can be done in a very efficient way. Remark that, from the decomposition $l = l_1N + l_2$, it follows that, in order to generate the whole matrix $\M_\alpha$, the minimum length of the vectors generated from \eqref{e:recursion} is respectively $(l_{lim}+1/2)N + 1$, $(l_{lim} + 1)N$ and $(l_{lim} + 1)N$.

Finally, we perform the sum over $l_2$ in \eqref{e:thanot1truncated}. Since
$$
e^{i2l_2s_{N-1-j}} = e^{-i2l_2s_j}, \qquad e^{i2l_2s_{j+N}} = e^{i2l_2s_j},
$$

\noindent it is enough to compute \eqref{e:thanot1truncated}, for with $j\in\{0, \ldots, N/2-1\}$, and extend the results until $j = 2N-1$, by means of those symmetries. Alternatively, it is possible to use the FFT, too.

Let us finish this section by mentioning that the case $\alpha = 1$ in \eqref{e:tha1} presents no difficulty. When $k$ is even, it is trivial to implement, and when $k$ is odd, we factorize and truncate $l$, as when $\alpha\not=1$, obtaining
\begin{equation}
\label{e:tha1truncated}
(-\Delta)^{1/2}(e^{iks_j})
\approx
\left\{
\begin{aligned}
& \frac{|k|\sin^2(s_j)}{L}e^{iks_j}, & \text{$k$ even,}
	\\
& \frac{ik}{L\pi}\left(\frac{-2}{k^2-4} - \sum_{l_2=-N/2}^{N/2-1}\left[\sum_{l_1=-l_{lim}}^{l_{lim}}b_{k,l_1,l_2}\right]e^{i2l_2s_j} \right), & \text{$k$ odd,}
\end{aligned}
\right.
\end{equation}

\noindent with
$$
b_{k,l_1,l_2} = \frac{4(-1)^{l_1} \sgn(l_1N+l_2)}{(k-2(l_1N+l_2))((k-2(l_1N+l_2))^2-4)}.
$$

\section{Numerical tests}

\label{s:numtests}

We have first considered two functions with polynomial decay:
$$
u_1(x) = \frac{x^2 - 1}{x^2 + 1}, \qquad u_2(x) = \frac{2x}{x^2 + 1},
$$

\noindent where the first one tends to $1$ as $\mathcal O(1/x^2)$, and the second one tends to $0$ as $\mathcal O(1/x)$. Under the change of variable $x = \cot(s)$, with $L = 1$, these functions become respectively $u_1(s) = \cos(2s)$ and $u_2(s) = \sin(2s)$, i.e., the real and imaginary parts of $e^{i2s}$, so we have computed the fractional Laplacian of $u_1(x)+i\,u_2(x)$. Using \textsc{Mathematica} applied to \eqref{intlema1}, and further simplifying the result by hand, we get
$$
(-\Delta)^{\alpha/2}(u_1(x) + i\,u_2(x)) = -\frac{2\Gamma(1+\alpha)}{(1+i\,x)^{1+\alpha}},
$$

\noindent or, in the $s$ variable,
\begin{equation}
\label{e:ei2sexact}
(-\Delta)^{\alpha/2}(e^ {i2s}) = -\frac{2\Gamma(1+\alpha)}{(1+i\cot(s))^{1+\alpha}} = -2\Gamma(1+\alpha)(-i\sin(s)e^{is})^{1+\alpha};
\end{equation}

\noindent note that, when $\alpha = 1$, we recover \eqref{e:tha1}. In general, it seems possible to compute explicitly $(-\Delta)^{\alpha/2}(e^{iks})$ when $k$ is even, at least for small values of $k$, although the complexity of the resulting expressions quickly grows with $k$, and it does not follow an evident pattern. On the other hand, we have been unable to obtain a compact formula for $(-\Delta)^{\alpha/2}(e^{iks})$, when $k$ is odd.

Taking different values of $N$ and $l_{lim}$, we have approximated numerically $(-\Delta)^{\alpha/2}(e^{i2s})$, which we denote as $[(-\Delta)^{\alpha/2}]_{num}(e^{i2s})$, by means of \eqref{e:thanot1truncated}, without generating the whole matrix $\M_\alpha$, for $\alpha \in \{0.01, 0.02, \ldots, 1.99\}$, except for the case $\alpha = 1$, which is trivial (altogether, $1998$ different values of $\alpha$). Then, we have compared the results with their exact value of $[(-\Delta)^{\alpha/2}]_{exact}(e^{i2s})$ given by \eqref{e:ei2sexact}, and computed the discrete $L^\infty$-norm of the error as a function of $\alpha$. In Table \ref{t:errorsei2s}, we show the maximum global value of the norm considering all $\alpha$, i.e.,
\begin{align}
\label{e:errorei2s}
\max_\alpha & \|[(-\Delta)^{\alpha/2}]_{num}(e^{i2s}) - [(-\Delta)^{\alpha/2}]_{exact}(e^{i2s})\|_\infty
	\cr
& = \max_\alpha\max_j\left|[(-\Delta)^{\alpha/2}]_{num}(e^{i2s_j}) - [(-\Delta)^{\alpha/2}]_{exact}(e^{i2s_j})\right|,
\end{align}

\noindent for different values of $N$ and $l_{lim}$. For comparison purposes, we also offer $l_{total}\equiv (2l_{lim}+1)N$, which is the exact number of values of $l$ taken, i.e., $l \in\{-l_{total}/2, \ldots, l_{total}/2-1\}$. The results reveal that the value of $l_{lim}$ necessary to achieve an error of the order of $5\cdot10^{-13}$ slowly decreases as $N$ is doubled, but, more importantly, the accuracy of the method does not deteriorate, as $N$ increases.
\begin{table}[!htbp]
	\centering
	\begin{tabular}{|c||c|c|c||c|c|c|}
		\hline $N$ & $l_{lim}$ & $l_{total}$ & Error & $l_{lim}$ & $l_{total}$ & Error
		\\
		\hline $4$ & $300$ & $4244$ & $4.8893\cdot10^{-12}$ & $530$ & $4244$ & $5.0268\cdot10^{-13}$
		\\
		\hline $8$ & $240$ & $6888$ & $4.9423\cdot10^{-12}$ & $430$ & $6888$ & $4.8097\cdot10^{-13}$
		\\
		\hline $16$ & $200$ & $11536$ & $4.8413\cdot10^{-12}$ & $360$ & $11536$ & $4.9461\cdot10^{-13}$
		\\
		\hline $32$ & $170$ & $19232$ & $4.5606\cdot10^{-12}$ & $300$ & $19232$ & $5.0138\cdot10^{-13}$
		\\
		\hline $64$ & $140$ & $32064$ & $4.9236\cdot10^{-12}$ & $250$ & $32064$ & $5.0184\cdot10^{-13}$		
		\\
		\hline $128$ & $120$ & $53888$ & $4.5427\cdot10^{-12}$ & $210$ & $53888$ & $5.0219\cdot10^{-13}$
		\\
		\hline $256$ & $100$ & $92416$ & $4.6956\cdot10^{-12}$ & $180$ & $92416$ & $5.0570\cdot10^{-13}$
		\\
		\hline $512$ & $80$ & $154112$ & $5.7013\cdot10^{-12}$ & $150$ & $154112$ & $5.0823\cdot10^{-13}$
		\\
		\hline $1024$ & $70$ & $287744$ & $4.8721\cdot10^{-12}$ & $140$ & $287744$ & $5.2887\cdot10^{-13}$
		\\
		\hline 
	\end{tabular}
	\caption{Maximum global error, given by \eqref{e:errorei2s}, between the numerical approximation of $(-\Delta)^{\alpha/2}(e^{i2s})$, given by \eqref{e:thanot1truncated}, and its exact value, given by \eqref{e:ei2sexact}. We have considered different values of $N$ and, for each $N$, a couple of values of $l_{lim}$. For comparison purposes, we also offer the total number of values of $l$ considered, $l_{total}\equiv (2l_{lim}+1)N$.}\label{t:errorsei2s}
\end{table}

In order to better understand the choice of $l_{lim}$ on the accuracy of the results, we have approximated $(-\Delta)^{\alpha/2}(e^{i2s})$, for $l_{lim} \in\{0, 1, \ldots, 1000\}$, and have plotted in Figure \ref{f:errorei2s} the corresponding maximum global error given by \eqref{e:errorei2s}. As we can see, the errors quickly decay from $l_{lim} = 0$, with an error of $3.1960\cdot10^{-3}$, to $l_{lim} = 210$, with an error of $5.0219\cdot10^{-13}$, from which it remains constant up to infinitesimal variations. This is important, because it shows that \eqref{e:thanot1truncated} is numerically stable, even for larger values of $l_{lim}$. A practical consequence of this is that, in case of doubt, it is safe to take a rather large value of $l_{lim}$.
\begin{figure}[!htbp]
\centering
\includegraphics[width=0.5\textwidth, clip=true, align=t]{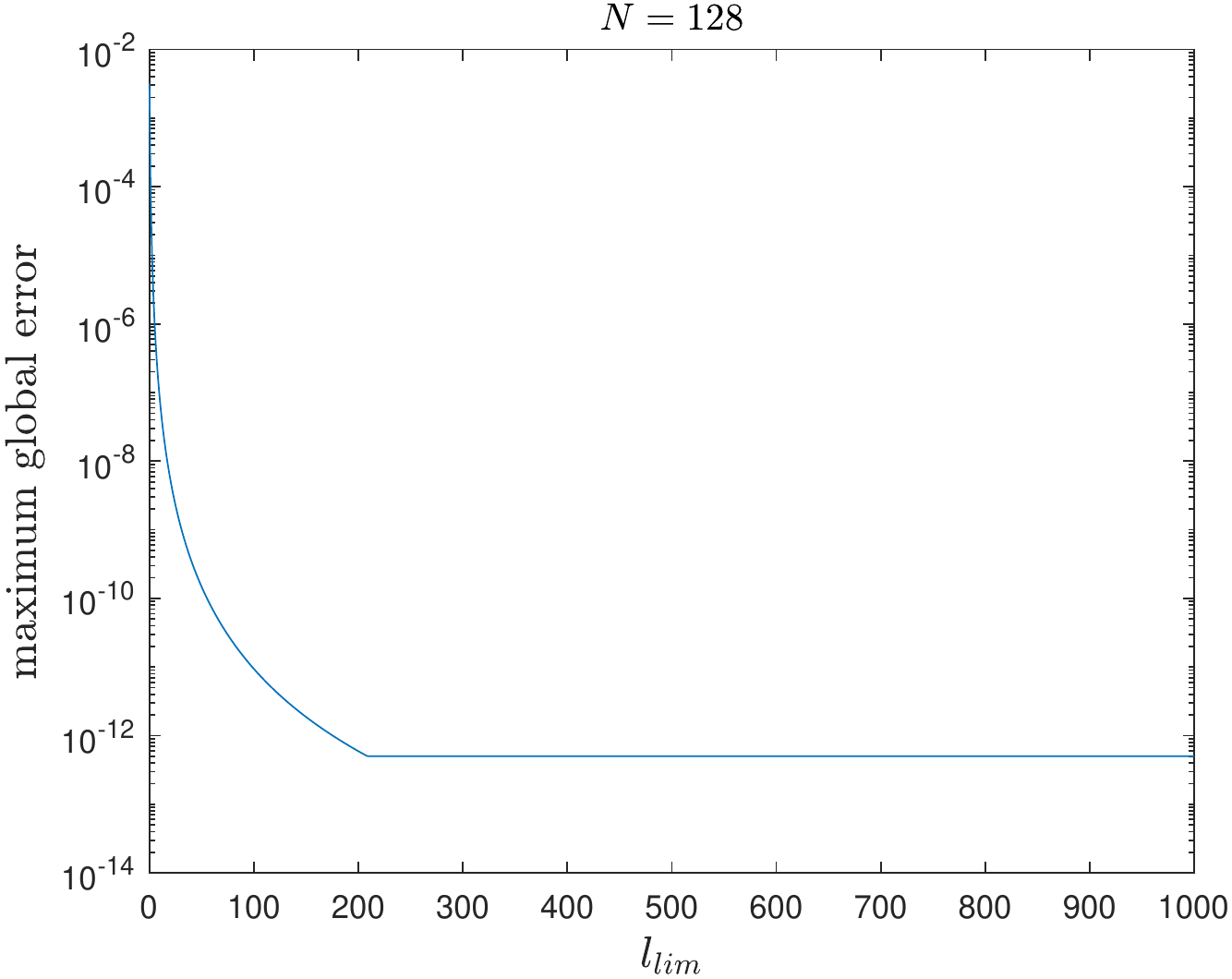}
	\caption{Maximum global error for $N = 128$, as a function of $l_{lim}$.}
\label{f:errorei2s}	
\end{figure}

Let us consider now a function with Gaussian decay,
$$
u_3=\exp(-x^{2}),
$$

\noindent such that (see, for instance, \cite[pp. 29-30]{pozrikidis}).
\begin{equation}
\label{e:Dav3exact}
(-\Delta)^{\alpha/2}u_3(x)=\frac{2^{\alpha}\Gamma(1/2+\alpha/2)}{\sqrt{\pi}}
{}_{1}F_{1}(1/2+\alpha/2,\ 1/2,\ -x^{2}),
\end{equation}

\noindent where $_{1}F_1$ is the Kummer confluent hypergeometric function, which can be evaluated accurately, among others, by \textsc{Matlab} (with the command \texttt{hypergeom}) and \textsc{Mathematica} (with the command \texttt{Hypergeometric1F1}), even if its evaluation is extremely time-consuming.

In this example, we have generated the whole matrix $\M_\alpha$ applied to the Fourier expansion \eqref{uenfourier2} of $u_3(s)$, as in \eqref{e:Ma}. Remember that, since $s\in[0,\pi]$, we have to extend it to $s\in[0,2\pi]$. In general, the most common option is an even extension at $s = \pi$, which yields a function that is at least continuous in $s\in[0,2\pi]$, and can be represented as a cosine series in $s$. However, in some cases, there are extensions that are smoother than the even one (see v.g. \cite{morton2009} and \cite{hybrechs2010}), causing the Fourier coefficients in \eqref{uenfourier2} to decay faster. This is not a minor point, because, even if $(-\Delta)^{\alpha/2}(e^{iks})$ can be computed accurately as we have seen in the previous example, the overall quality of the results depends also on the adequacy of the representation \eqref{uenfourier2}.

In this example, since $u_3(x)$ tends to zero as $x\to\pm\infty$ (or $s\to0^+$ and $s\to\pi^-$), we have considered both an even and an odd extension at $s = \pi$, i.e., such that $u_3(\pi^{+}) = u_3(\pi^{-})$ and $u_3(\pi^{+}) = -u_3(\pi^{-})$, respectively. For this function, in the even case, we also have that $u_3(s + \pi) = u_3(s)$, which implies that only even frequencies appear in \eqref{uenfourier2}; whereas in the odd case, $u_3(s + \pi) = -u_3(s)$, so only odd frequencies appear in \eqref{uenfourier2}. As a consequence, besides comparing two types of extensions, we are also testing the even and odd cases in \eqref{e:thanot1truncated} and \eqref{e:tha1truncated}.

We have approximated $(-\Delta)^{\alpha/2}u_3(x)$ for $\alpha\in\{0.01, 0.02, \ldots, 1.99\}$ (including the case $\alpha = 1$), for $L = 1$, $l_{lim} = 500$, and different values of $N$. In Table \ref{t:errorsgauss}, we give the maximum global errors computed as in \eqref{e:errorei2s}. As we can see, the even extension provides only slightly better results, and the errors quickly decays, as $N$ increases.
\begin{table}[!htbp]
	\centering
	\begin{tabular}{|c|c|c|c||c|c|c|}
		\hline $N$ & Error (even) & Error (odd)
		\\
		\hline $4$ & $3.8426\cdot10^{-1}$ & $4.8492\cdot10^{-1}$
		\\
		\hline $8$ & $1.1222\cdot10^{-1}$ & $1.3210\cdot10^{-1}$
		\\
		\hline $16$ & $1.4269\cdot10^{-2}$ & $1.7825\cdot10^{-2}$
		\\
		\hline $32$ & $4.0393\cdot10^{-4}$ & $4.7926\cdot10^{-4}$
		\\
		\hline $64$ & $1.4351\cdot10^{-6}$ & $1.6891\cdot10^{-6}$
		\\
		\hline $128$ & $1.5947\cdot10^{-10}$ & $1.8755\cdot10^{-10}$
		\\
		\hline $256$ & $8.3982\cdot10^{-12}$ & $2.5453\cdot10^{-11}$
		\\
		\hline 
	\end{tabular}
	\caption{Maximum global error, between the numerical approximation of $(-\Delta)^{\alpha/2}(e^{-x^2})$, and its exact value, given by \eqref{e:Dav3exact}, for different values of $N$, considering an even extension and an odd extension. $l_{lim} = 500$.}\label{t:errorsgauss}
\end{table}

Even if we have taken so far $L = 1$, this is usually by no means the best option, as we can see in Figure \ref{f:errorgaussLevenodd}, where we have plotted the maximum global error for $N = 64$, and $L\in\{0.1, 0.2, \ldots, 10\}$. The results for the even extension and the odd extension are again similar, and the best errors are achieved in both cases at $L = 4.6$, and are respectively $3.8400\cdot10^{-13}$ and $3.9466\cdot10^{-13}$. Therefore, a good choice of $L$ can improve drastically the accuracy of the results.

\begin{figure}[!htbp]
	\centering
	\includegraphics[width=0.5\textwidth, clip=true, align=t]{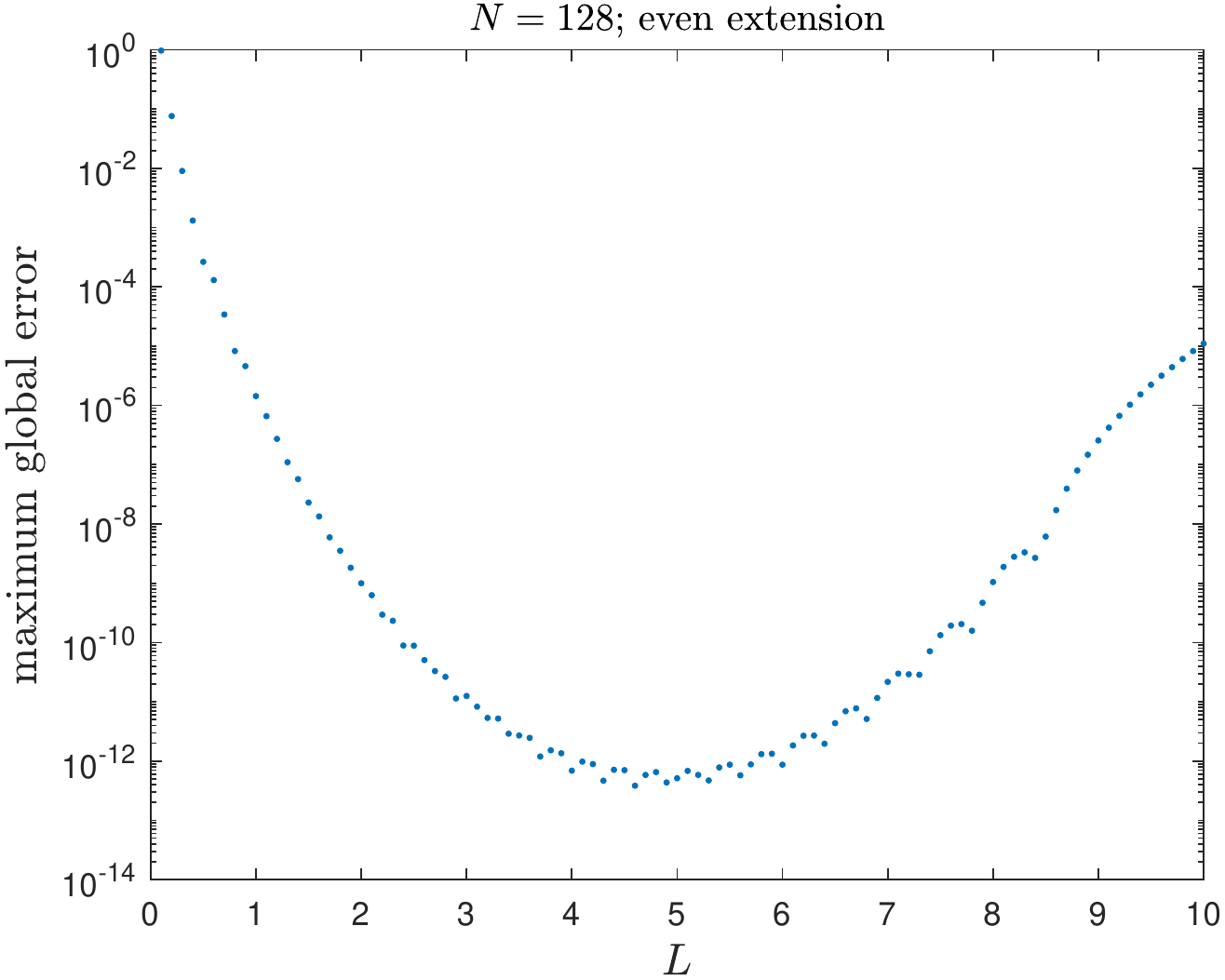}\includegraphics[width=0.5\textwidth, clip=true, align=t]{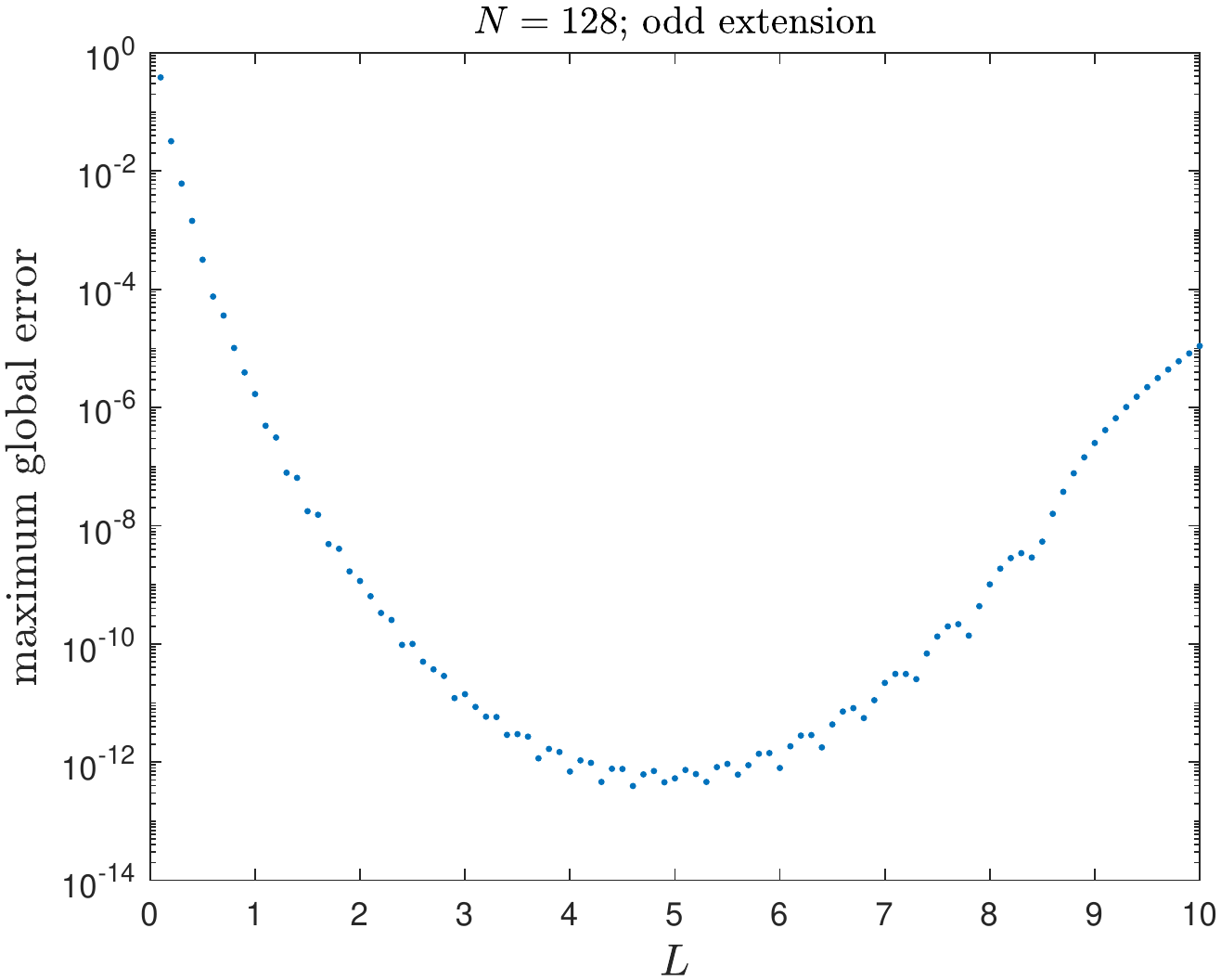}
	\caption{Maximum global error for $N = 64$, and $L\in\{0.1, 0.2, \ldots, 1\}$, considering an even extension and an odd extension.}
	\label{f:errorgaussLevenodd}
\end{figure}

Although there are some theoretical results \cite{Boyd1982}, the optimal value of $L$ depends on more than one factor: number of points, class of functions, type of problem, etc (see also \cite{delahozvadillo,delahozcuesta2016}). For instance, in the case of $(-\Delta)^{\alpha/2}$, the best choice of $L$ might depend on $\alpha$, too. However, a good working rule of thumb seems to be that the absolute value of a given function at the extreme grid points is smaller than a threshold. On the other hand, the Fourier representation \eqref{uenfourier2}, together with \eqref{e:xxc}, makes straightforward to change $L$ (and $x_c$ or $N$). Let us recall that, given a function $u(x)$, we are considering a spectral interpolant such that
\begin{equation}
\label{e:u(x)L}
u(x) \approx \sum_{k=-N}^{N-1}\hat u(k)e^{ik\arccot((x-x_c)/L)},
\end{equation}

\noindent and, to determine $\{\hat u(k)\}$ we ask \eqref{e:u(x)L} to be an equality at the nodes $x_j = x_c+L\cot(s_j)$, which yields \eqref{uenfourier4}. Therefore, if we choose new values of $L$ and $x_c$, say $L_{new}$ and $x_{c,new}$, and want to approximate $u(x)$ at the corresponding nodes $x_{new,j} = x_{c,new}+L_{new}\cot(s_j)$ by using \textit{spectral interpolation}, it is enough to evaluate the right-hand side of \eqref{e:u(x)L} at those nodes:
$$
u(x_{new,j}) \approx \sum_{k=-N}^{N-1}\hat u(k)e^{ik\arccot((x_{c,new}-x_c+L_{new}\cot(s_j))/L)},
$$

\noindent where, when $0\le j \le N - 1$, we consider the $\arccot$ function to be defined in $[0,\pi)$, and, when $N\le j\le 2N-1$, to be defined in $[\pi,2\pi)$. Moreover, from $\{u(x_{new,j})\}$, we obtain the corresponding $\{\hat u_{new}(k)\}$ by using again a pseudospectral approach, i.e., by imposing that \eqref{e:u(x)L} with the updated $L_{new}$ and $x_{c,new}$ is an equality at $x = x_{j,new}$, for all $j$:
$$
u(x_{new,j}) = \sum_{k=-N}^{N-1}\hat u_{new}(k)e^{ik\arccot((x_{new,j}-x_{c,new})/L_{new})},
$$

\noindent so the coefficients $\hat u_{new}(k)$ are given by \eqref{uenfourier4}, introducing  $u(x_{new,j})$ in the place of $u(s_j)$, and taking $L_{new}$ and $x_{c,new}$. Then,
$$
u(x) \approx \sum_{k=-N}^{N-1}\hat u_{new}(k)e^{ik\arccot((x-x_{c,new})/L_{new})},
$$

\noindent Finally, \eqref{e:u(x)L} allows also changing $N$; e.g., if $N$ is increased, we just add some extra $\hat u(k)$ equal to zero; if it is decreased, we remove some $\hat u(k)$. In all the cases considered, it is important to choose the new values of $L$, $x_c$ and $N$, in such a way that there is no loss of accuracy.

\section{Numerical experiments for (\ref{fishereq})}

\label{s:fisher}

As an illustration of the method presented in Section \ref{s:ComputationFracLap}, we will simulate numerically the one-dimensional nonlinear evolution equation (\ref{fishereq}) in the monostable case, i.e., with the following nonlinear source term:
\begin{equation}
f(u) = u(1-u).
\label{sourceterm1}
\end{equation}

\noindent The zeros of (\ref{sourceterm1}) correspond to stationary states, $u \equiv 1$ being stable and $u \equiv 0$ being unstable. We recall that in the local case with  $\alpha=2$, which corresponds to the integer-order Laplacian, a solution of (\ref{fishereq}) may take the form of a traveling wave front as $t$ increases, thus traveling with constant speed. Such fronts travel to the right, approximating $u=1$, as $x \to-\infty$, and $u=0$, as $x\to\infty$. Such solutions settle as the stable state invades the unstable one \cite{kendall}. The wave speed depends on the decay in the tail of the initial condition, and is greater for slower decays (see e.g. \cite{needham}, \cite{aronson2} and others). 

For $\alpha\in(0,2)$, we expect front solutions that travel to the right, invading $u=0$ soon after initiating the evolution, with an initial condition $u_0(x)$ that satisfies $u_0(x)\to 1$, as $x\to-\infty$, and $u_0(x)\to 0$, as $x\to\infty$. Unlike in the local case, these fronts do not travel with constant speed, but with a speed that increases exponentially with $t$ (see \cite{CabreRoquejoffre2013}); more precisely, one expects $c(t)\sim \exp(\sigma\, t)$ with $\sigma=f'(0)/\alpha$ or faster for slow decaying initial conditions, and $\sigma=f'(0)/(1+\alpha)$, for fast decaying ones. 

We restrict ourselves to the example of slow decaying (according to \cite{CabreRoquejoffre2013}) initial conditions; more precisely, we consider
\[
u(x,0) = \left(\frac{1}{2} - \frac{x}{2\sqrt{1+x^{2}}}\right)^{\alpha}.
\]

\noindent In order to check that, for $\alpha\in(0,2)$, the propagation has indeed speed that increases exponentially with time, we track the evolution of $x_{0.5}(t)$, which denotes the value of $x$ such that $u(x,t) = 0.5$. and gives an approximation of the position of the front. To obtain $x_{0.5}$, we apply a bisection method: we find the value of $j$ for which $u(x_{j+1}) < 0.5 < u(x_j)$; then, we approximate $u((x_{j}+u_{j+1})/2)$ by spectral interpolation, etc., until convergence is achieved.

\begin{figure}[!htbp]
	\begin{center}
		\includegraphics[width=0.496\textwidth, clip=true, align=b]{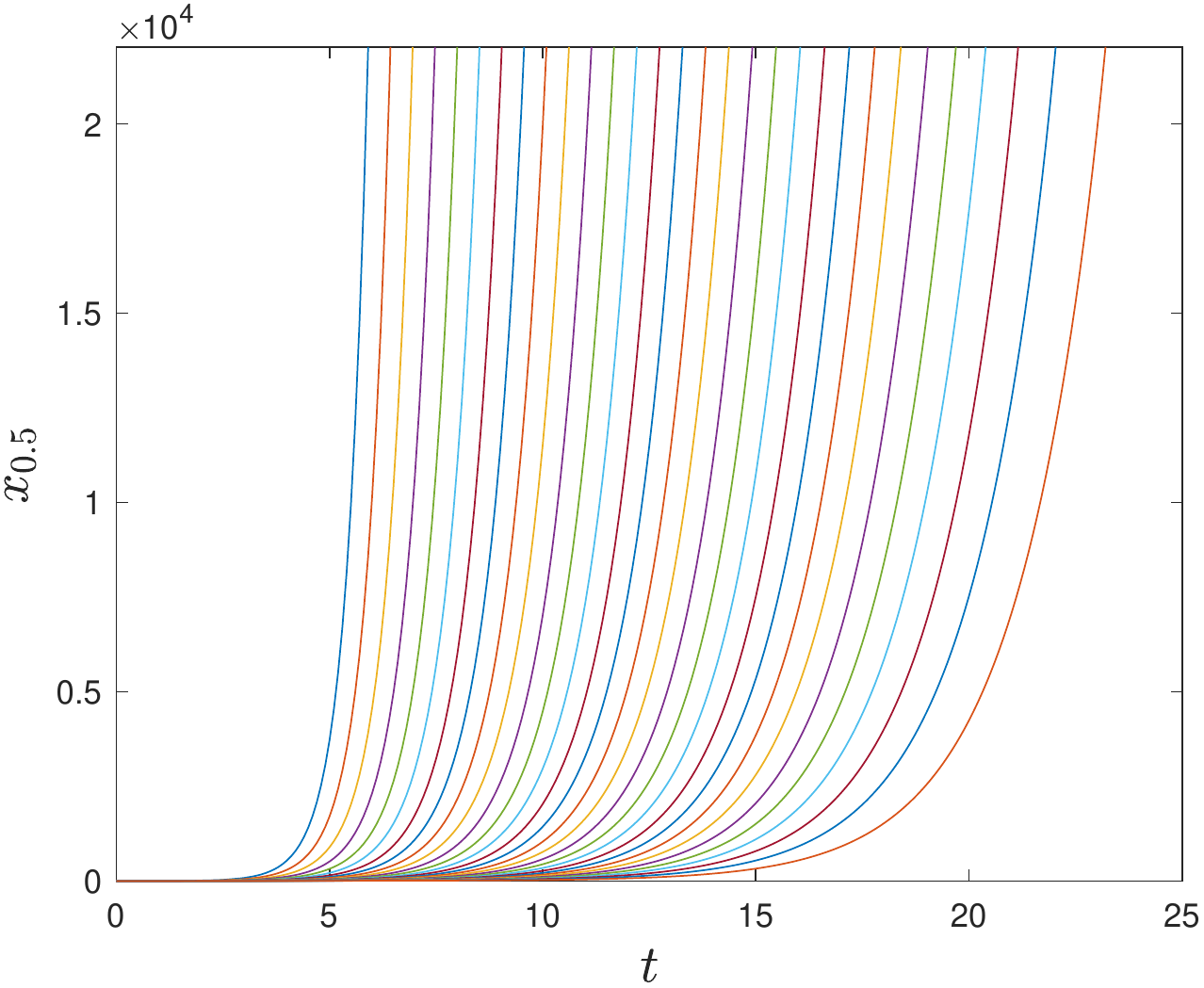}\includegraphics[width=0.504\textwidth, clip=true, align=b]{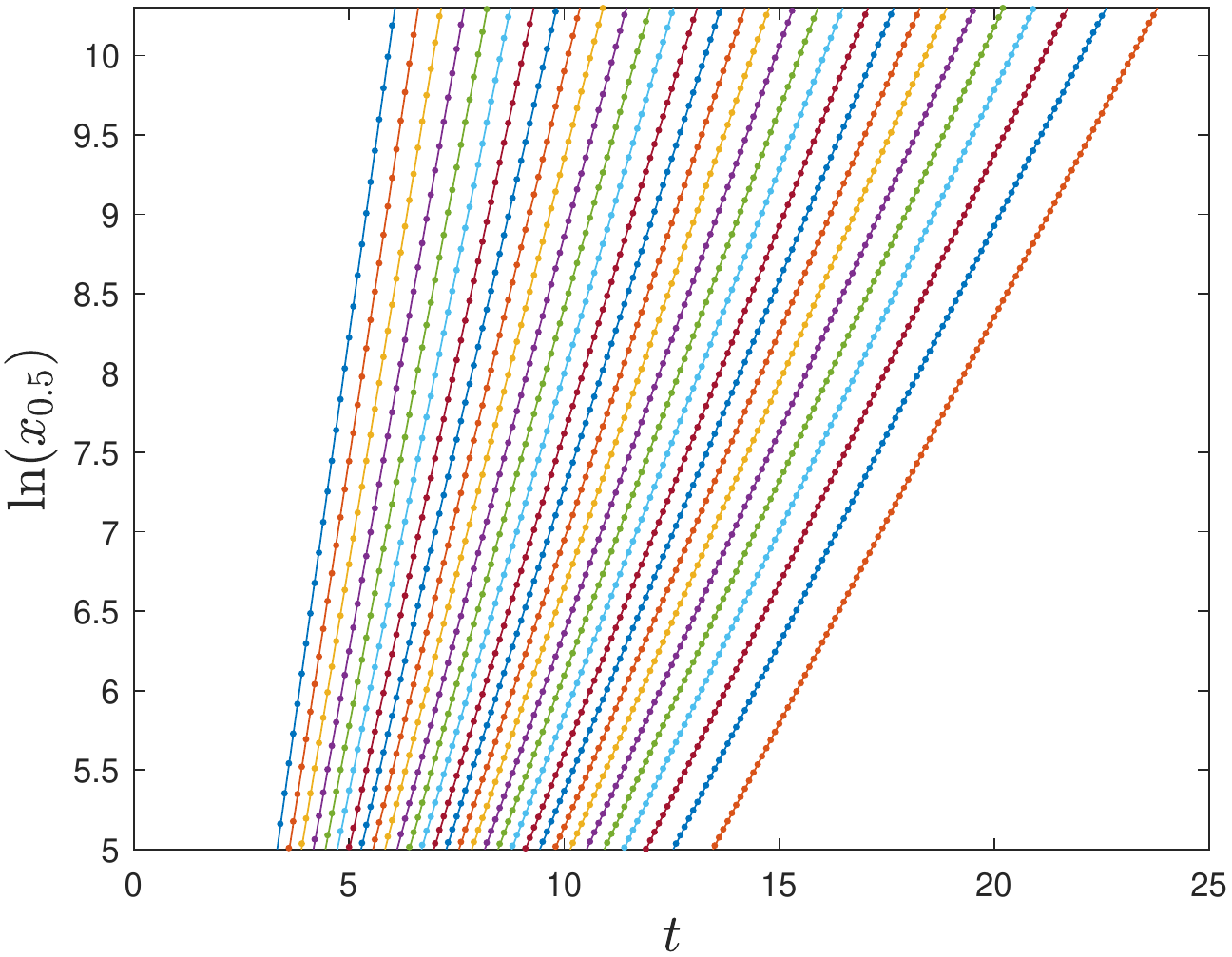}
		\caption{$\alpha=0.5, 0.55, \ldots, 1.95$, $L=10^3 / \alpha^3$, $\Delta t = 0.01$ and $N=1024$. Left: $x_{0.5}(t)$ against $t$. Right: $\ln(x_{0.5}(t))$ against $t$, and the corresponding least-square fitting lines. In both subfigures, the curves are ordered according to $\alpha$: the left-most ones correspond to $\alpha = 0.5$, and the right-most ones, to $\alpha = 1.95$.}\label{f:fittingx05vst}
	\end{center}
\end{figure}

In all the numerical experiments, we have considered an even extension at $s = \pi$, which is enough for our purposes, taken $l_{lim} = 500$ in \eqref{e:thanot1truncated}, and used the classical fourth-order Runge-Kutta scheme (see for instance \cite[p. 226]{shampine1997}) to advance in time. We have done the numerical simulation for $\alpha = 0.5, 0.55, \ldots, 1.95$, taking $\Delta t = 0.01$ and $N = 1024$. Since the exponential behavior of $x_{0.5}(t)$ appears earlier for smaller $\alpha$, larger values of $L$ appear to be convenient in that case. In this example, after a couple of trials, we have found that taking $L = 1000 / \alpha^3$ produces satisfactory results. On the left-hand side of Figure \ref{f:fittingx05vst}, we have plotted $x_{0.5}(t)$ against $t$. On the right-hand side of Figure \ref{f:fittingx05vst}, we have plotted $\ln(x_{0.5}(t))$ against $t$, omitting the initial times, so the exponential regime is clearly observable; in all cases, the points are separated by time increments of $0.1$, and, for each value of $\alpha$, the accompanying line is precisely the least-square fitting line, which shows that the linear alignment is almost perfect.

In Figure \ref{f:fittingsigma05vsalpha}, we have plotted, with respect to $\alpha$ the slopes of the least-square fitting lines corresponding to the right-hand side of Figure \ref{f:fittingx05vst}, which we denote as $\sigma_{0.5}$; observe that the colors of the stars are in agreement with their corresponding curves in Figure \ref{f:fittingx05vst}. We have also plotted the curve $1/\alpha$, using a dashed-dotted black line. The results show that the agreement of $\sigma_{0.5}$ with respect to $1/\alpha$ improves, as $\alpha\to2^-$: on the one hand, when $\alpha = 0.5$, $\sigma_{0.5} = 1.9346$, and $1 / 0.5 = 2$; on the other hand, when $\alpha = 1.95$, $\sigma_{0.5} = 0.51277$, and $1 / 1.95 = 0.51282$. Therefore, the numerical experiments seem to suggest that
$$
x_{0.5}(t) \sim e^{\sigma_{0.5}t} \sim e^{t / \alpha} \Longrightarrow c(t) \approx x_{0.5}'(t) \sim e^{t / \alpha},
$$

\noindent which is in good agreement with \cite{CabreRoquejoffre2013}, because, from \eqref{sourceterm1}, $f'(0) = 1$.

\begin{figure}[!htbp]
	\begin{center}
		\includegraphics[width=0.5\textwidth, clip=true, align=b]{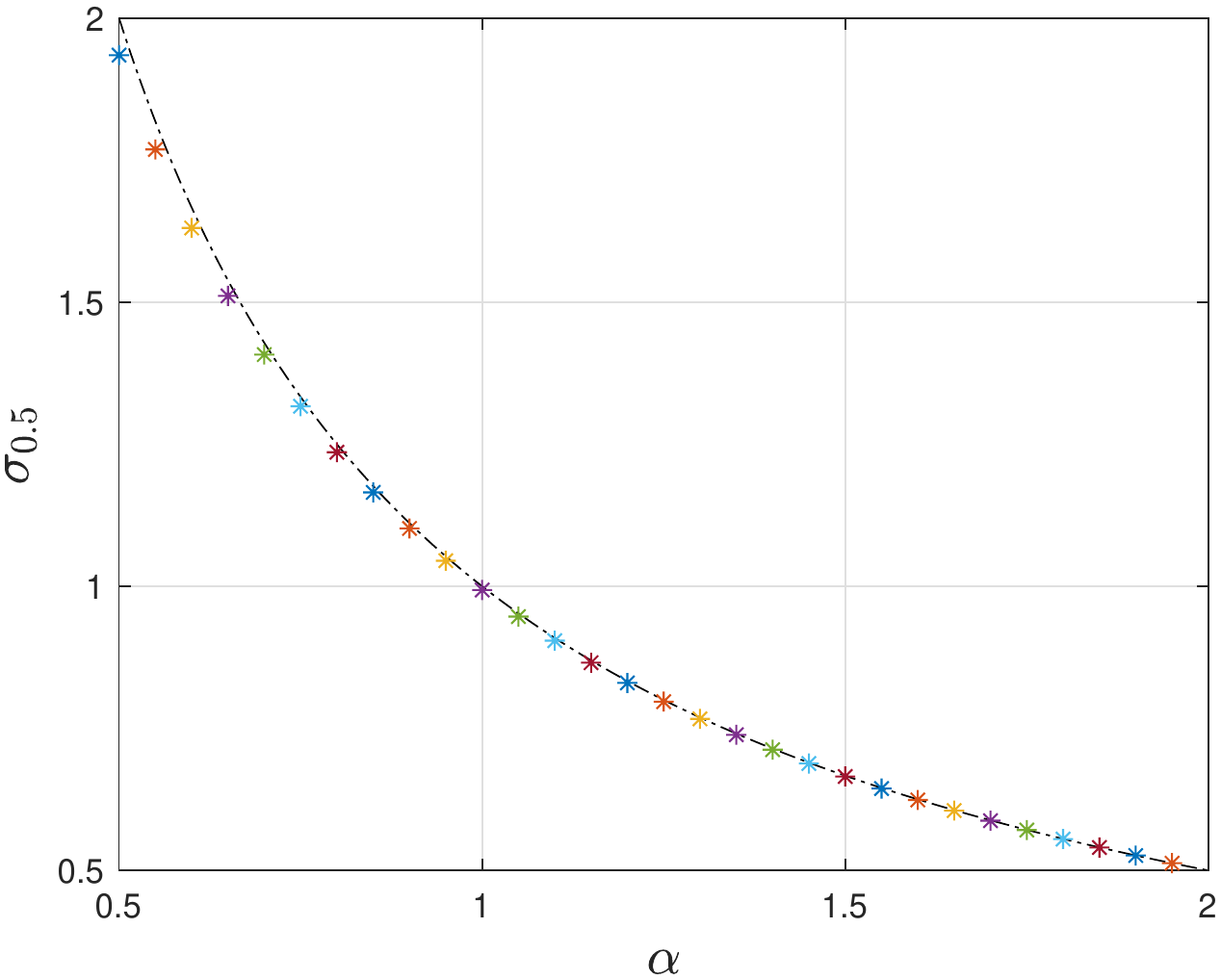}
		\caption{Slopes of the least-square fitting lines, as obtained in the right-hand side of Figure \ref{f:fittingx05vst}; the colors of the stars are in agreement with their corresponding curves in Figure \ref{f:fittingx05vst}. The dashed-dotted black curve is the plot of $1/\alpha$.}\label{f:fittingsigma05vsalpha}
	\end{center}
\end{figure}

In order to see whether the results for $\alpha = 0.5$ can be improved, we have repeated the simulations for that case, taking $L = 10000$, $\Delta t = 0.01$, $N = 8192$. Even if, at first sight, these values could be deemed excessive, they are not, because we are able to reach $t = 9$, instant at which $x_{0.5}(9)$ is greater than $10^7$. Indeed, in order to capture accurately the exponential behavior, it is convenient to advance until times as large as possible. On the left-hand side of Figure \ref{fitunacurvaalpha}, we have plotted $x_{0.5}(t)$, for $t\in[0,9]$; on the right-hand side, $\ln(x_{0.5}(t))$, for $t\in[5,9]$, obtaining again an almost perfect linear fitting. Furthermore, in this case, $\sigma_{0.5} = 1.9865$, which is remarkably closer to the predicted value $1/0.5=2$ than in Figure \ref{f:fittingsigma05vsalpha}. Therefore, in order to approximate accurately $\sigma_{0.5}$ for values of $\alpha$ smaller than $0.5$, it will be convenient to take even larger values of $N$ and $L$.

\begin{figure}[!htbp]
	\begin{center}
		\includegraphics[width=0.5\textwidth, clip=true, align=b]{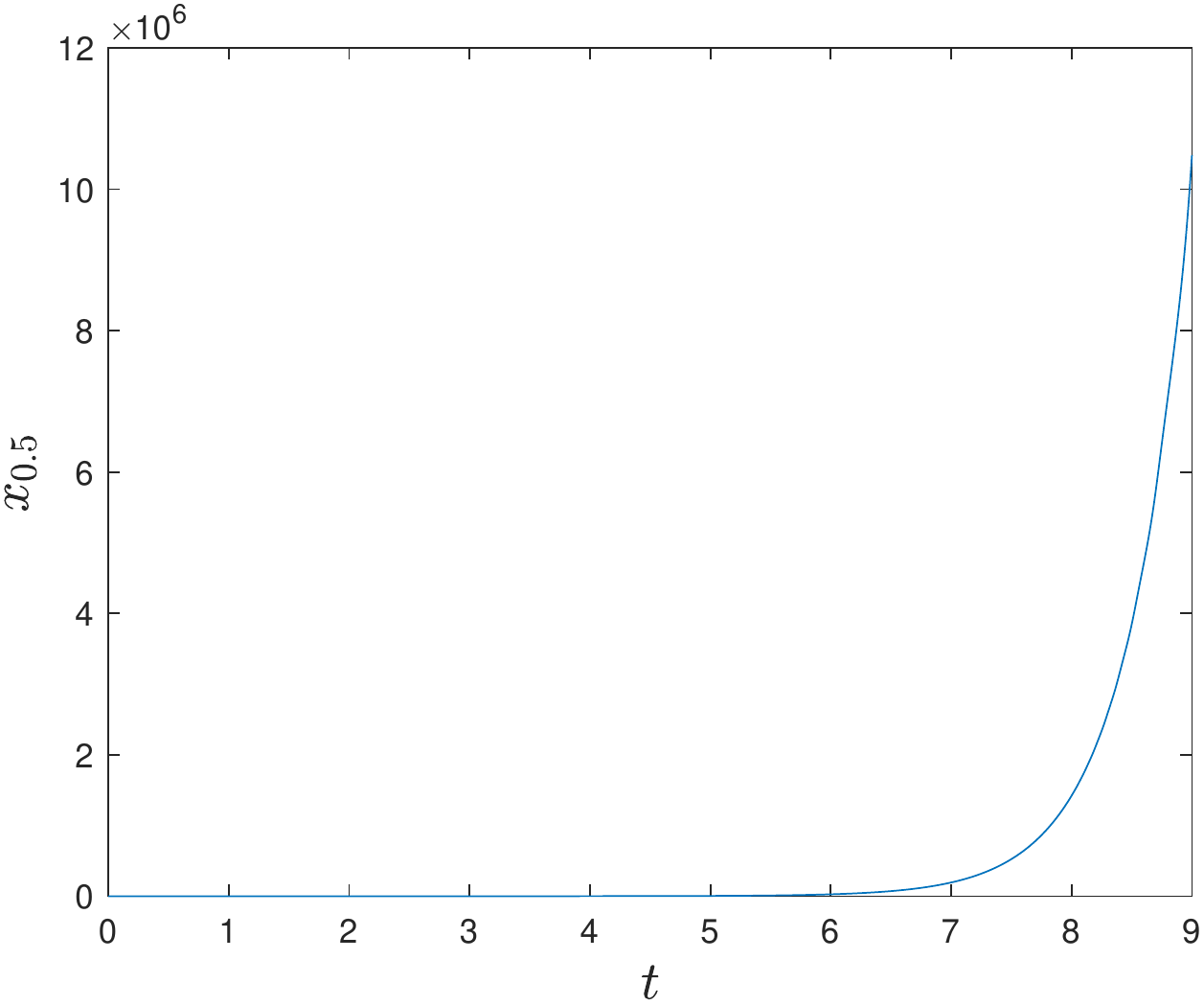}\includegraphics[width=0.5\textwidth, clip=true, align=b]{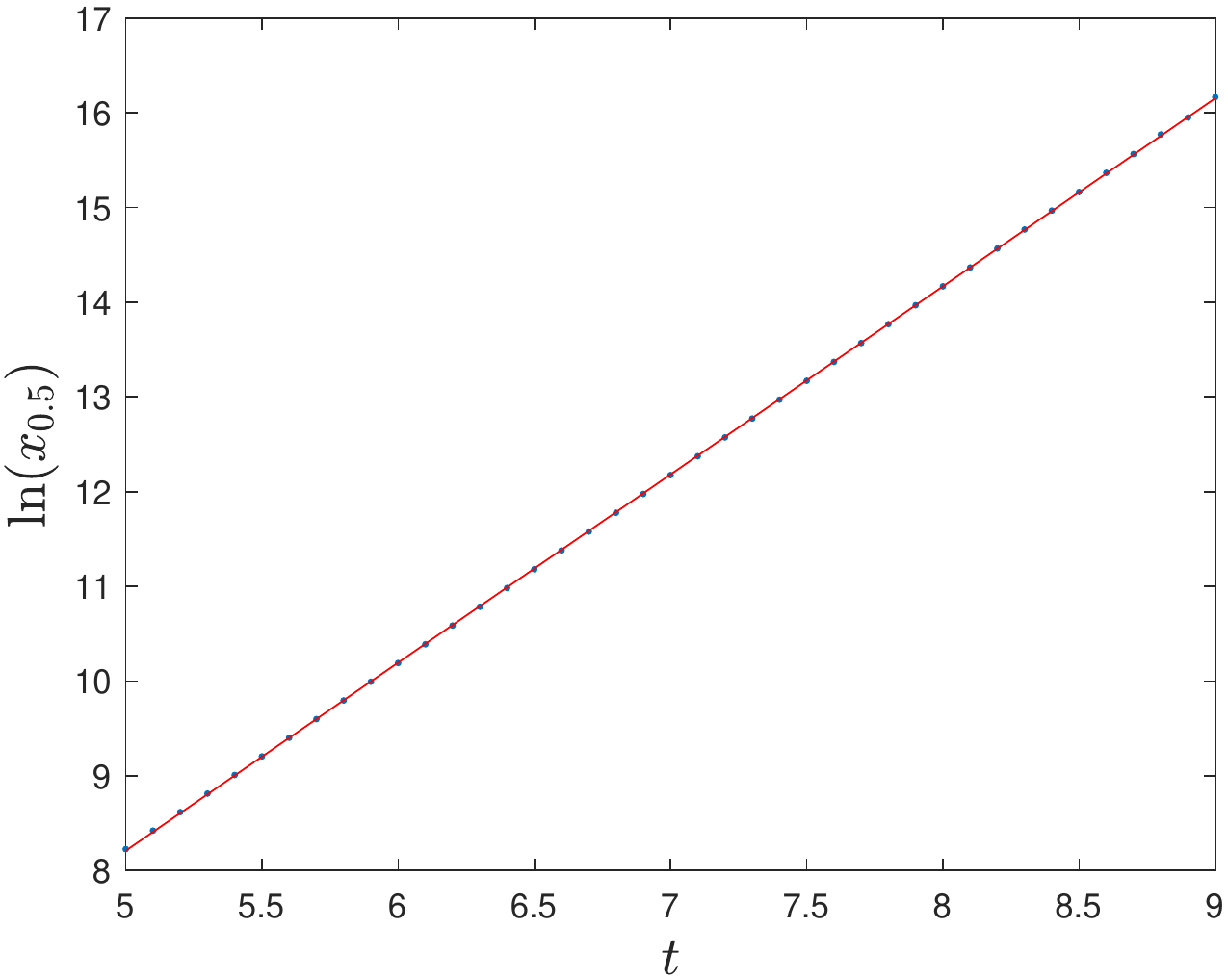}
		\caption{$\alpha=0.5$, $L=10^4$, $\Delta t = 5\cdot10^{-3}$, and $N=8192$. Left: $x_{0.5}(t)$ against $t\in[0,9]$. Right: $\ln(x_{0.5}(t))$ against $t\in[5,9]$, and the corresponding least-square fitting line.}\label{fitunacurvaalpha}
	\end{center}
\end{figure}
	
\section*{Acknowledgments}
The authors acknowledge the financial support of the Spanish Government through the MICINNU projects MTM2014-53145-P and PGC2018-094522-B-I00, and of the Basque Government through the Research Group grants IT641-13 and IT1247-19. J. Cayama also acknowledges the support of the Spanish Government through the grant BES-2015-071231.

\end{document}